\documentclass[letterpaper,10pt]{article}

\usepackage{enumerate}
\usepackage{amsmath}
\usepackage[english]{babel}
\usepackage{amssymb}
\usepackage{epsfig}
\usepackage{pstricks,pst-node,pst-coil,pst-plot,pst-text}
\usepackage{ifpdf}
\usepackage{subfigure}
\usepackage{caption}
\usepackage{amsthm}
 \usepackage{wrapfig}
\usepackage{sidecap}
\newtheorem{theorem}{Theorem}
\newtheorem{corollary}[theorem]{Corollary}

\newtheorem{definition}{Definition}

\newtheorem{lemma}[theorem]{Lemma}
\newtheorem{proposition}[theorem]{Proposition}

\newtheorem{observation}{Observation}

\newlength{\saveparindent}
\newcommand{\pp}{\textbf{p}}

\usepackage{graphicx,color}
\title{Directed cycle double covers: structure and generation of hexagon graphs}

\author{
  Andrea~Jim\'enez\thanks{
  Supported by CNPq (Proc.~477203/2012-4) and FAPESP (Proc.~2011/19978-5).}\\
\small Instituto of Matem\'atica e Estat\'istica \\[-0.8ex]
\small Universidade de S\~ao Paulo \\[-0.8ex]
\small \texttt{ajimenez@ime.usp.br}\\
\and
  Mihyun~Kang\thanks{Partially supported by the German Research Foundation (KA 2748/2-1 and KA 2748/3-1).}\\
\small Institut f\"{u}r Optimierung und Diskrete Mathematik  \\[-0.8ex]
\small Technische Universit\"{a}t Graz \\[-0.8ex]
\small \texttt{kang@math.tugraz.at}\\
  \and
  Martin~Loebl\thanks{Partially supported by the Czech Science Foundation under the contract number P202-13-21988S.}\\
\small Department of Applied Mathematics \& \\[-0.8ex]
\small  Institute for Theoretical Computer Science \\[-0.8ex]
\small Charles University \\[-0.8ex] 
\small \texttt{loebl@kam.mff.cuni.cz}
}

\date{}

\begin{document}

\maketitle
\begin{abstract}
Jaeger's directed cycle double cover conjecture can be formulated as a problem of
  existence of special perfect matchings in a class of graphs that we call hexagon graphs.
In this work, we explore the structure of hexagon graphs.
We show that hexagon graphs are braces that can be generated 
from the ladder on 8 vertices using two types of McCuaig's augmentations. 
\end{abstract}

\section{Introduction}\label{sec:intro} 
The long-standing Jaeger's directed cycle double cover conjecture~\cite{Jaeger19851},
usually known as DCDC conjecture, is broadly considered to
be among the most important open problems in graph theory.
A typical formulation asks whether every 2-connected graph admits a family of 
cycles such that one may prescribe an orientation on each cycle of the family
in such a way that each edge $e$ of the graph belongs to exactly two cycles and 
these cycles induce opposite orientations on $e$. 
In order to prove the DCDC conjecture, a wide variety of 
approaches have arisen~\cite{Jaeger19851, Zhang97integerflows}, 
among them, the topological approach. 
The topological approach claims that 
the DCDC conjecture is equivalent to the statement that every
cubic bridgeless graph admits an embedding in a closed orientable surface 
such that every edge belongs to exactly two distinct face boundaries
defined by the embedding; that is, with no dual loop.

In this work, we formulate the DCDC
  conjecture as a problem  of existence of special perfect matchings in a class 
  of graphs that we call hexagon graphs. Initially, our motivation for the formulation
  of the DCDC conjecture on hexagons are critical embeddings~\cite{kenyon,mercat},
  that in particular are embeddings with no dual~loop.
  
The main goal of this work is to discuss recent progress on the 
  study of the structure of hexagon graphs.
  The class of hexagon graphs of cubic bridgeless graphs turns out to be a subclass of braces.
The class of braces, along with bricks, are a fundamental class of graphs in matching theory,
  mainly because they are building blocks of a perfect matching decomposition procedure;
  namely of the tight cut decomposition procedure~\cite{Lovasz:1987:MSM:30820.30826}.
In~\cite{McCuaig98}, McCuaig introduced a method for 
  generating all braces starting from a large base set of graphs and recursively 
  making use of 4 distinct types of operations.
In this paper, we show that hexagon graphs are braces that can be generated 
  from the ladder on 8 vertices using 2 types of McCuaig's operations.

In the following, we make precise the notions discussed above and
  formally state our main result.

\subsection{Hexagon graphs}

Hexagon graphs are the main ingredient and the center of attention of this work.
In this section, we define the class of hexagon graphs, look over some of its fundamental properties and
  formulate the DCDC conjecture as a question about this new class of graphs. 
Despite our original motivation for this new formulation of the DCDC conjecture are critical embeddings, 
  in this work we do not introduce this notion, and we present the details and proofs regarding the formulation using 
  rotation systems of graphs, a well known and convenient combinatorial representation of embeddings 
  on closed orientable surfaces~~\cite[\S 3.2]{opac-b1131924}.
The advantage of using rotation systems is that we 
  avoid topological arguments and present the equivalence to the DCDC conjecture
  in a purely combinatorial way.

We refer to the complete bipartite graph $K_{3,3}$ as a \emph{hexagon} and
say that a bipartite graph $H$ has a hexagon $h$ if $h$ is a subgraph of $H$.
For a graph $G$ and a vertex $v$ of $G$, let $N_G(v)$ denote the set of neighbors of $v$ in $G$.

\begin{definition}[Hexagon Graphs]\label{def:heaxgraphs} Let $G$ be a cubic graph with vertex set $V$ and edge set $E$.
A hexagon graph of $G$ is a graph $H$ obtained from $G$ following the next rules:
\begin{enumerate}
\item We replace each vertex $v$ in $V$ by a hexagon $h_v$ of $H$ so that  
for every pair $u$, $v \in V$, if $u \neq v$, then $h_u$ and $h_v$ are vertex disjoint. 
      Moreover,  $V(H) = \{V(h_v): v \in V \}$.
\item For each vertex $v \in V,$ let  $\{v_i: i \in \mathbb{Z}_6\}$ denote the vertex set of $h_v$ 
      and $\{v_{i}v_{i+1}, v_{i}v_{i+3}: i \in \mathbb{Z}_6\}$ its edge set. 
      With each neighbor $u$ of $v$ in $G$, we associate an index $i_{v(u)}$ 
      from the set $\{0,1,2\} \subset \mathbb{Z}_6$ so that if 
      $N_G(v)=\{u,w,z\}$, then $i_{v(u)}$, $i_{v(w)}$, $i_{v(z)}$ are pairwise distinct. 
\item \label{def.hex_third} See Figure~\ref{fig:hexagons}. Let $X= \cup_{v\in V} \{v_{2i} :i\in \mathbb{Z}_6\}$ and 
      $Y= \cup_{v\in V} \{v_{2i+1} :i\in \mathbb{Z}_6\}$.
      We replace each edge $uv$ in $E$ by two vertex disjoint edges $e_{uv}$, $e'_{uv}$ so that 
      if both $v_{i_{v(u)}}$, $u_{i_{u(v)}}$ belong to either $X$ or $Y$, then
      $e_{uv}= v_{i_{v(u)}} u_{i_{u(v)}+3}$, $e'_{uv}= v_{i_{v(u)}+3} u_{i_{u(v)}}$.
      Otherwise, 
      $e_{uv}= v_{i_{v(u)}} u_{i_{u(v)}}$, $e'_{uv}= v_{i_{v(u)}+3} u_{i_{u(v)}+3}$.
      Moreover, $E(H)=\{E(h_v): v \in V\} \cup \{ e_{uv}, e'_{uv} : uv \in E\}$.
      \end{enumerate} 
\end{definition}

\begin{figure}[h] 
 \centering 
 \ifpdf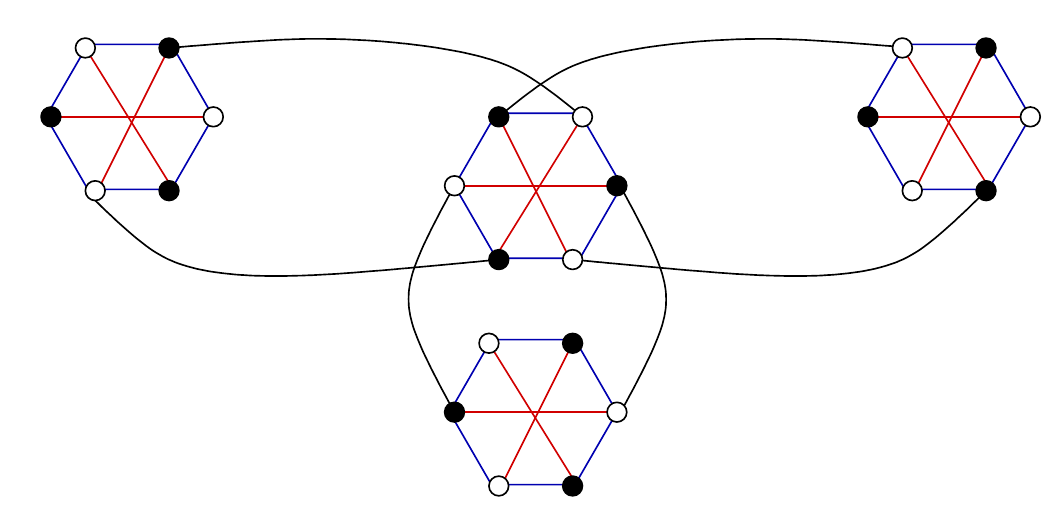\else\input{hexagons.ps_t}\fi
 \caption{Local representation of the hexagon-neighborhood of a hexagon $h_v$ in a hexagon graph $H$
 of a cubic graph $G$. The hexagon $h_v$ is associated with vertex $v$, where $N_{G}(v)=\{u,w,z\}$.
 Red edges are depicted as red lines, blue edges are depicted as blue lines and 
 white edges as black lines. The set $X$ is represented by filled-in white vertices
 and the set $Y$ by filled-in black vertices. Moreover, $i_{v(u)} =0$, $i_{v(w)} =1$, $i_{v(z)} =2$, $i_{u(v)}=0$,
 $i_{w(v)}=2$ and $i_{u(v)}=2$.}
\label{fig:hexagons}
\end{figure}

We say that $h_v$ is the hexagon of $H$ associated with the vertex $v$ 
  of $G$ and that $\{h_v: v \in V \}$ is the \emph{set of hexagons of $H$}. 
For $uv \in E$, we say that $h_u$ and $h_v$ are \emph{hexagon-neighbors} in $H$.	
We shall refer to the set of edges $\bigcup_{v\in V} \{v_{i}v_{i+3}: i\in\mathbb{Z}_6\}$ as the 
set of \emph{red edges of $H$}, to the set of edges $\{ e_{uv}, e'_{uv} : uv \in E\}$
as the set of \emph{white edges of $H$} and finally, to the set of edges
$\bigcup_{v\in V} \{v_{i}v_{i+1}: i\in\mathbb{Z}_6\}$ as the set of \emph{blue edges of $H$}
 (see Figure~\ref{fig:hexagons}). Moreover, we shall 
 say that a perfect matching of $H$ containing only blue edges is a \emph{blue perfect matching}.  

\begin{observation}
\label{o.bip}
Hexagon graphs of cubic graphs are bipartite.
\end{observation} 
\begin{proof} 
Let $H$ be a hexagon graph of a cubic bridgeless graph. Let $X$, $Y$ be the sets
defined in Definition~\ref{def:heaxgraphs}, item~\textit{\ref{def.hex_third}}. Note that $\{X, Y\}$
is a partition of $V(H)$ and that there are no edges connecting vertices
of the same partition class.
\end{proof}

The following two observations are straightforward.

\begin{observation}
\label{o.desc}
Let $G$ be a cubic graph and $H$ be a hexagon graph of $G$. The following properties hold.
\begin{enumerate}
\item $H$ is a 4-regular graph.
\item No white edge of $H$ connects two vertices of the same blue hexagon. 
\item Both, the set of red edges of $H$ and the set of white edges of $H$ form a perfect matching of $H$.
\item Let $|V(G)|$ denote the cardinality of $V(G)$. There are $2^{|V(G)|}$ distinct blue perfect matchings.
\end{enumerate}
\end{observation}

\begin{observation}
\label{o.iso}
If $H$ and $H'$ are hexagon graphs of a cubic graph $G$, then $H$ and $H'$ are isomorphic.
\end{observation}

\subsubsection*{Rotation systems and embeddings without dual loops}

Recall that our goal in this section is to reformulate the following statement:
every cubic bridgeless graph admits an embedding on a closed orientable surface 
without dual loops. For this purpose, we now introduce a 
combinatorial representation of embedding of graphs on closed orientable surfaces; namely \emph{rotation systems}. 

Let $G$ be a graph. For each $v \in V(G)$, let $\pi_{v}$ be a cyclic permutation of the
edges incident with $v$. 
A collection $\pi = \{\pi_{v} : v \in V(G)\}$ is called a 
\emph{rotation system} of $G$. The proof of the following statement
can be found in~\cite[\S 3.2]{opac-b1131924}.  

\begin{theorem} 
Let $\pi$ be a rotation system of a graph $G$. Then $\pi$ encodes an embedding of $G$ 
on a closed orientable surfaces with set of face boundaries 
\begin{equation}\label{eq:faceboundaries} \{e_1e_2 \cdots e_k: \, e_i=v^iv^{i+1}
\in E(G), \, 
\pi_{v^{i+1}}(e_i)= e_{i+1}, 
\, e_{k+1}=e_1 \,\, \text{and $k$ minimal}\}. \end{equation}
Moreover, the converse holds. That is, every embedding of $G$ on a
closed orientable surface defines a rotation system $\pi$ of $G$
where the set of face boundaries is given by the set described in \emph{(\ref{eq:faceboundaries})}.
\end{theorem}

In~Lemma~\ref{lemma:hexagon-embedding}, we state that blue perfect matchings of hexagon 
  graphs of a cubic graph $G$ define embeddings of $G$ on closed orientable surfaces with  distinguished set of face boundaries,
  and vice versa.
The proof is based on a natural bijection between blue perfect matchings and rotation systems. 
We first need to make an observation. 

\begin{observation}
Let $M$ be a blue perfect matching of $H$ and let $W$ be the set of white edges of $H$. Each cycle 
$C$ in $M\Delta W$ induces a subgraph in $G$ defined by the set of edges
$ \{uv \in E(G): e_{uv} \in C \, \text{or} \,\, e'_{uv} \in C\}$.
\end{observation}

\begin{theorem}\label{lemma:hexagon-embedding} 
Let $G$ be a cubic graph, $H$ be the hexagon graph of $G$ and $W$ be the set of white edges of $H$.
Each blue perfect matching $M$ of $H$ encodes 
an embedding of $G$ on a closed orientable surface with a set of face boundaries,
the set of subgraphs of $G$ induced by the cycles in $M\Delta W$. 
Moreover, the converse holds. That is, each embedding of $G$ on a closed orientable surface 
defines a blue perfect matching $M$ of $H$, where the set of subgraphs of $G$ induced by all cycles in $M\Delta W$
coincides with the set of face boundaries of the embedding.
\end{theorem}
\begin{proof} 
It suffices to prove that there is a bijective function $f$ from the set of blue perfect matchings of $H$
 to the set of rotation systems of $G$ such that for every blue perfect matching $M$ of $H$, 
 the set of subgraphs of $G$ induced by the cycles in $M\Delta W$
 equals the set of subgraphs described in~(\ref{eq:faceboundaries}) defined by the rotation system $f(M)=\pi$.
  
Let $v \in V(G)$, $N_{G}(v)=\{u,w,z\}$, and without loss of generality (by Observation~\ref{o.iso}) 
we assume that $i_{v(u)}=0$, $i_{v(w)}=1$ and $i_{v(z)}=2$. Let $M$ be a blue perfect matching of $H$. 
The restriction of $M$ to $h_v$ is either $\{v_{0}v_{1},v_{2}v_{3}, v_{4}v_{5}\}$ or 
$\{v_{1}v_{2},v_{3}v_{4}, v_{5}v_{0}\}$. If the restriction is $\{v_{0}v_{1},v_{2}v_{3}, v_{4}v_{5}\}$,
then the cyclic permutation of the edges incident with $v$ in the rotation system $f(M)=\pi$ of $G$
is $\pi_{v}= (uv\,\,wv\,\,zv)$. 
Otherwise, the cyclic permutation is given by $\pi_{v}= (uv\,\,zv\,\,wv)$. 
It is a routine to check that $f$ is the desired bijection.
\end{proof}
 
The following result is crucial for our approach.

\begin{proposition}\label{lemma:dualloops}
Let $G$ be a cubic graph, $H$ be the hexagon graph of $G$, $M$ be a blue perfect matching of $H$ and $W$ be the set of white edges of $H$.
 The embedding of $G$ encoded by $M$ has a dual loop if and only if there is a cycle
 in $M\Delta W$ that contains the end vertices of a red edge.
\end{proposition}
\begin{proof}
An embedding of $G$ has a dual loop if and only if there is an edge $uv \in E(G)$ that belongs to exactly  
one face boundary, say $C'$. The face boundary $C'$ is a subgraph of $G$ 
induced by a cycle $C$ of $M\Delta W$. We have $C'$ is the only subgraph induced 
by a cycle of $M\Delta W$ that contains $uv$ if and only if $e_{uv}$ and $e'_{uv}$
belong to $C$. The lemma follows.
\end{proof}

Motivated by Proposition~\ref{lemma:dualloops}, we shall say that a blue perfect matching $M$ is {\em safe} if no cycle of 
$M\Delta W$ contains the end vertices of a red edge. 
In Corollary~\ref{thm.one} we establish the formulation of the DCDC Conjecture on hexagon graphs.
Note that the result of Corollary~\ref{thm.one} follows directly
from Theorem~\ref{lemma:hexagon-embedding} and Proposition~\ref{lemma:dualloops}.

\begin{corollary}
\label{thm.one}
A cubic graph $G$ has a directed cycle double cover if and only if its hexagon graph $H$ admits a safe 
perfect matching.
\end{corollary}

\subsection{Braces}\label{subsec:braces}
A \emph{brace} is a simple (that is, no loops and no multiple edges), connected, bipartite graph 
  on at least six vertices, and 
  with a perfect matching such that for every pair of nonadjacent edges, there is a perfect 
  matching containing the pair of edges. In~\cite{McCuaig98}, McCuaig presented a method 
  for generating braces. He showed that all braces can be constructed from a base set using four operations.  
  In the following we describe McCuaig's method for generating braces.

Let $H$ be a bipartite graph and $x$ be a vertex of $H$ of degree at least 4. 
  Let $N_1, N_2$ be a partition of $N_H(x)$ such that $|N_1|,|N_2| \geq 2$. 
  Let $\{x^1, v, x^2 \}$ be a set of vertices such that $\{x^1, v, x^2 \} \cap V(H)=\emptyset$.
The \emph{expansion of $x$ to $x^1vx^2$}, or briefly an \emph{expansion of $x$} is 
  the operation composed of the following three steps: (i) delete $x$, (ii) add the new 
  path $x^1vx^2$, and (3) connect every vertex of $N_1$ ($N_2$, respectively) to the vertex $x^1$ 
  ($x^2$, respectively). For $i \in \{1,2\}$, we say that $N_i$ is the partition associated with $x^i$.
Note that if $H'$ is a graph obtained from $H$ by the expansion of a vertex, 
  then $H'$ is also bipartite.\\ 
  \emph{Augmentations.} If $H'$ is a bipartite graph obtained 
  from $H$ by adding a new edge, then we say that $H'$ is obtained from $H$ by a \emph{type-1 augmentation}. 
  Let $x$ and $w$ be two vertices in the same partition class of $H$ 
  such that $x$ has degree at least $4$. 
  If $H'$ is obtained from $H$ expanding $x$   
  to $x^1 v x^2$ and adding the new edge $vw$, then 
  we say that $H'$ is obtained from $H$ by a \emph{type-2 augmentation}. 
  Let $x$ and $y$ be two vertices of $H$ of distinct partition classes such that $d_H(x), d_H(y) \geq 4$.
  Let $H'$ be the bipartite graph obtained from $H$ by expanding $x$ and $y$   
  to $x^1 v x^2$ and $y^1 u y^2$ respectively, and adding the new edge $vu$.
  If $x$ and $y$ are not connected in $H$, the operation for obtaining $H'$ from $H$
  is called a \emph{type-3 augmentation}, otherwise it is called a \emph{type-4 augmentation}. 

\begin{figure}[h]
\centering
\subfigure[type-1 augmentation]
{
\ifpdf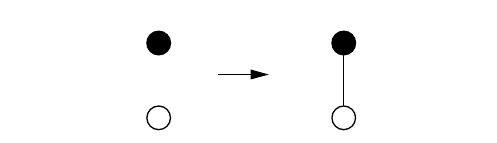\else\input{type_1.ps_t}\fi
  \label{fig:type1}
}
\subfigure[type-2 augmentation]
  {\ifpdf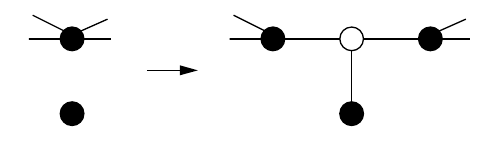\else\input{type_2.ps_t}\fi
  \label{fig:type2}
  }
  \caption{Simple augmentations}
\label{fig:simple_aug}
\end{figure}

  If $H'$ is obtained from $H$ by a type $i$ augmentation for some $i \in \{1,2,3,4\}$, then we say that
  $H'$ is obtained from $H$ by an \emph{augmentation}. If $i \in \{1,2\}$, 
  then we say that $H'$ is obtained from $H$ by a \emph{simple augmentation} (see Figure~\ref{fig:simple_aug}).

Let $\mathcal{B}$ be the infinite set consisting of all bipartite 
M\"{o}bius ladders, ladders and biwheels (see Figure~\ref{fig:base_set}). 

\begin{figure}[h]
\centering
\subfigure[M\"{o}bius ladders: $M_6, M_{10}, M_{14}, M_{18}, \ldots$]
{
\ifpdf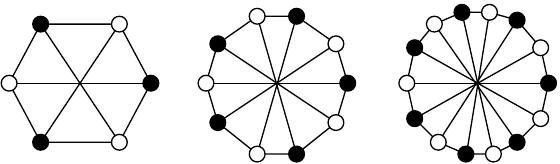\else\input{mobius_ladders.ps_t}\fi
  \label{fig:mobius_ladders}
}\qquad \qquad
\subfigure[Ladders: $ L_{8}, L_{12}, L_{16}, L_{20}, \ldots $]
  {\ifpdf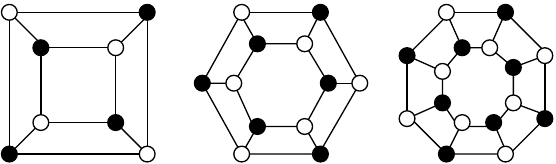\else\input{ladders.ps_t}\fi
  \label{fig:ladders}
  }
  
  \subfigure[Biwheels: $ B_{10}, B_{12}, B_{14}, B_{16}, \ldots $]
  {\ifpdf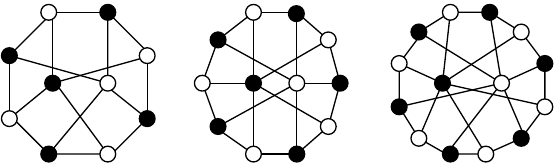\else\input{biwheels.ps_t}\fi
  \label{fig:biwheels}
  }
\caption{The base set $\mathcal{B}$.}
\label{fig:base_set}
\end{figure}

\begin{theorem}[McCuaig, 1998] 
Let $H$ be a bipartite graph. Then  $H$ is a brace if and only if there exists a sequence
$H_0, H_{1}, \ldots, H_{k}$ of bipartite graphs such that $H_0 \in \mathcal{B}$, $H_{i}$
may be obtained from $H_{i-1}$ by an augmentation for each $i\in \{1, \ldots, k\}$ and $H_k=H$. 
\end{theorem}

\subsection{Main results}\label{algo}

The main results of this paper are the following.

\begin{theorem}
\label{thm.brace}
Let $G$ be a cubic graph. Then the hexagon graph $H$ of $G$ is a brace if and only if $G$ is bridgeless.
\end{theorem}
\begin{proof}
Let $B$, $W$, and $R$ denote the set of blue, white, and
red edges, respectively. Moreover,
a blue edge is denoted by $b$, a white edge by $w$, and a red edge by $r$. 
Each pair of disjoint edges, $\{b, b'\}$, $\{r,r'\}$, or $\{b,r\}$, can be simply
extended to a perfect matching of $H$. 

We note that each component of $W\cup R$ is a cycle on four vertices, a {\em square}.
Let $w, w'$ be a pair of disjoint white edges. The edges $w, w'$ belong to the same 
square of $W\cup R$, or to two different squares of $W\cup R$. 
In either case $w, w'$ can be naturally extended to a perfect matching of $H$.
Similarly,  each edge of a pair $w, r$ of disjoint white and red edges belongs to different squares of $W\cup R$, and
therefore it can be completed into a perfect matching of $H$.

Finally we consider a pair $b, w$ of disjoint white and blue edges. If the hexagon with $b$ does not contain
an end vertex of $w$, then it is not difficult to extend $b, w$ to a perfect matching of $H$. 
Hence, let $h_u$ be the hexagon that contains $b$ and an end vertex of $w$,
and let $h_v$ be the hexagon that contains the other end vertex of $w$.
Let  $b=u_{i}u_{i+1}$, $w=u_kv_j$, where $i, j, k \in \mathbb{Z}_6$.

If $k \notin \{i+3, i+4\}$, then $b, w$ can be completed into a perfect matching of $H$
that contains the edges $b, w$, and $u_{i+3}u_{i+4}$.

Hence, without loss of generality we can assume that $k=i+3$.
Let $e_{uv}= u_{i}v_{j+3}$ and $e_{uz}= u_{i+1}z_{l}$ (notation as in Definition~\ref{def:heaxgraphs}.{\emph{3}}),
where $z$ is the neighbor of $v$ in $G$ such that the white edge with an end vertex $u_{i+1}$
has an end vertex in $h_z$, and $l \in \mathbb{Z}_6$.
Given that in $G$, edges $uv, uz$ have a common end vertex $u$ represented by hexagon $h_u$, 
edge $b=u_{i}u_{i+1}$ can be seen as {\em the transition} between $uv, uz$, 
  while $u_{k}u_{k+1}$ can be seen as this transition reversed. 

Now let $G$ be bridgeless. We observe that two adjacent edges in a cubic bridgeless graph belong to a common cycle. 
Let $C$ be such a cycle for $uv, uz$.  

The two possible orientations of $C$ correspond to two disjoint cycles 
$C_b, C_w$ in $H$, where $b\in C_b$ and $w \in C_w$;
they contain the transition and transition reversed (between $uv, uz$), respectively.
Let $M_b$ be the perfect matching of $C_b$ consisting of all blue edges
and $M_w$ be the perfect matching of $C_w$ consisting of all white edges.
In particular, $b\in M_b$ and $w\in M_w$. 
Since each hexagon of $H$ is intersected by $C_b \cup C_w$ either in 
a pair of disjoint blue edges, or in the empty set, $M_b \cup M_w$ can be extended to a perfect matching
of $H$.

On the other hand, if $G$ has a bridge $e= \{u,v\}$, then let $V_1$ be the component of $G- e$ containing $u$. 
Any perfect matching of $G$ extending $b, w$ must induce
a perfect matching of $\cup_{x\in V_1}h_x\setminus \{u_{i+3}\}$, 
but this set consists of an odd number of vertices
and thus no perfect matching containing $b, w$ can exist.  
\end{proof}

\begin{theorem}\label{theo:main}
Let $G$ be a cubic bridgeless graph and $L_{8}$ denote the ladder on 8 vertices. 
There is a sequence $H_0, H_{1}, \ldots, H_{k}$ of bipartite graphs such that
$H_0 = L_{8}$, $H_{i}$ can be obtained from $H_{i-1}$  by a simple augmentation 
for each $i\in \{1, \ldots, k\}$  and $H_k$
is the hexagon graph of $G$. 
\end{theorem}

The crucial ingredients in the proof of Theorem~\ref{theo:main} are odd ear decompositions 
of cubic bridgeless graphs. We now give a rough sketch of the proof.
 Let $G$ be a cubic bridgeless graph, $H$ be its hexagon graph,
 and $(G_0, G_i, P_i)^l$ be an odd ear decomposition of $G$
 (see Subsection~\ref{eardecomp}).
With each intermediate subgraph $G_i$ of the odd ear decomposition of $G$ we associate
 an auxiliary graph $H'_i$. In particular, with (the cycle) 
 $G_0$ we associate the ladder $L_8$.
For each $i \in\{1,\ldots,l\}$, the auxiliary graph $H'_i$ contains the hexagons
 $h_v$ of $H$ such that $v$ has degree $3$ in $G_i$. Hence, $H'_l$ contains
 all hexagons of $H$ and indeed (by construction) it turns out to be isomorphic to $H$. 
The proof is based on the fact that for each $i \in\{1,\ldots,l\}$, it is possible to
generate $H'_i$ from $H'_{i-1}$ by a sequence of simple augmentations.

The rest of the paper is devoted to prove Theorem~\ref{theo:main}. 
The proof of Theorem~\ref{theo:main} is divided into two parts.
The first part is the generation of hexagon graphs from square graphs and
the second is the construction of square graphs from the ladder on $8$ vertices.
In Section~\ref{sec:prelim}, we introduce the concept of square graphs and 
 prove that hexagon graphs can be obtained from square graphs by a short sequence 
 of simple augmentations. Section~\ref{sec:construction} and Section~\ref{sec:proofsecondtheorem}
 focus on the construction of square graphs.

\section{Square graphs}\label{sec:prelim}

A \emph{square} is a complete bipartite graph on $4$ vertices, namely $K_{2,2}$. 
  We say that a bipartite graph has a square $s$ if it contains $s$ as a subgraph. 
Next we define \emph{square graphs}.

\begin{definition}[Square graphs]\label{def:square}
  Let $G$ be a cubic bridgeless graph with vertex set $V$ and edge set $E$.
Let $M$ be a perfect matching of $G$. An $M$-square graph of $G$ is a bipartite graph $Q$ with neither loops nor multiple edges
  satisfying the following properties:
\begin{enumerate}
\item For each vertex $v$ in $V,$ the graph $Q$ has a square $s_v$. If $u, v \in V$ are such that  
      $u \neq v$, then $s_v$ and $s_u$ are vertex disjoint subgraphs of~$Q$. 
      Moreover,  $V(Q) = \{V(v): v \in V \}$.
\item The set of edges of $Q$ is given by $$E(Q)=\{E(s_v): v \in V\} \cup \{\textbf{uv}: uv \in E\},$$ where 
      $\{\textbf{uv}: uv \in E\}$ is defined such that the following conditions hold:
      \begin{enumerate}
       \item For each edge $uv \in E$, there are edges $e_{u}$ in $E(s_{u})$ and 
	     $e_{v}$ in $E(s_{v})$ such that the subgraph of $Q$ induced
	     by the set of edges $\{e_u,e_v\} \cup \textbf{uv}$ is isomorphic to $K_{2,2}$. 
	     In particular, $|\textbf{uv}|=2$.
	     The edges $e_u$ and $e_v$ are called \emph{the supporting edges of} $\textbf{uv}$
	     in $s_u$ and $s_v$, respectively.     
      \item  Let $v \in V$ and $N_G(v)=\{u,w,z\}$. If $uv \in M$, then
            the supporting edges of $\textbf{wv}$ and $\textbf{zv}$ in $s_v$ are vertex disjoint.
         \end{enumerate}
\end{enumerate}
\end{definition}
We say that $s_v$ is the \emph{square associated} with vertex $v$ 
  and that $\{s_v: v \in V \}$ is the set of squares of~$Q$. 
  For each $uv \in E$, if $uv \in M$, then we say that $(s_u,s_v)$
  is a \emph{pair of matched squares of $Q$}. 
  Moreover, the subset of edges {{\bf\emph{uv}}} is called \emph{the projection
  of $uv$ in $Q$}.
    We usually denote by $\{v_i: i \in \mathbb{Z}_4\}$ the vertex set of the square $s_v$
      and by $\{v_iv_{i+1}: i \in \mathbb{Z}_4\}$ its edge set.  

Note that the graph obtained by contracting each square of $Q$ to a single point 
and then by deleting multiple edges is precisely $G$.
The following is a natural observation about square graphs.

\begin{observation} For every connected component $C$ of $G-M$ ($C$ is a cycle since $G$
is cubic), there exists a ladder $L$ on $4\cdot |C|$ vertices in the set of connected components 
of $Q-\{\textbf{e}: e\in M\}$ such that $v$ is a vertex of $C$ if and only if $s_v$ is a square of $L$.  
\end{observation}

In Lemma~\ref{theo:transition_braces}, we state that hexagon graphs
  can be generated from square graphs using simple augmentations. 

\begin{lemma}\label{theo:transition_braces}
  Let $G$ be a cubic bridgeless graph, $M$ be a perfect matching of $G$ and
  $Q$ be an $M$-square graph of $G$.
  Then there is a sequence of bipartite graphs $H_0,H_1,\ldots, H_l$ such that
  $H_0=Q$, $H_{i}$ may be obtained from $H_{i-1}$ by a simple augmentation for each $i\in \{1, \ldots, l\}$ 
  and $H_l$ is the hexagon graph of $G$. 
\end{lemma}

\begin{proof} We first describe an operation composed of a sequence of 
  simple augmentations which we apply to each pair of matched squares 
  in order to generate a pair of hexagon-neighbors; 
  we shall call this operation  a \emph{double augmentation.} 
Let $(s_u,s_v)$ be a pair of matched squares of $Q$. By definition, 
  all distinct configurations of the supporting edges of \textbf{\emph{uv}}
  in $s_u$ and $s_v$, respectively, are the ones depicted in Figure~\ref{fig:configurations_matchpairs}.   
\begin{figure}[h]
  \centering 
    \subfigure[]
  { 
    \ifpdf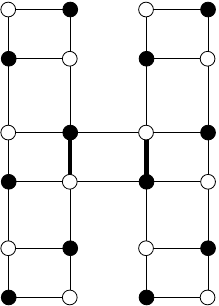\else\input{config_1.ps_t}\fi
    \label{fig:config_matchpairs_1}
  } \qquad
    \subfigure[]
  { 
    \ifpdf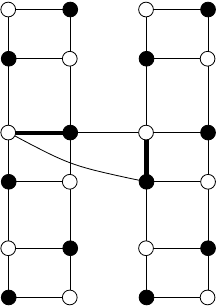\else\input{config_2.ps_t}\fi
    \label{fig:config_matchpairs_2}
  } \qquad
     \subfigure[]
  { 
    \ifpdf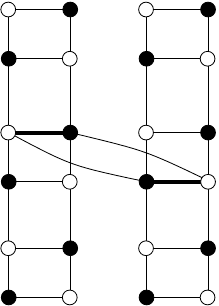\else\input{config_3.ps_t}\fi
    \label{fig:config_matchpairs_3}
  } 
  \caption{Possible locations of the supporting edges of \textbf{\emph{uv}}
  in $s_u$ and $s_v$ for a pair $(s_u,s_v)$ of matched squares of $Q$. Supporting edges are depicted by thick lines.}
  \label{fig:configurations_matchpairs}
  \end{figure}

We assume that the supporting edges of \textbf{\emph{uv}} for the pair $(s_u,s_v)$ are 
  configured as in Figure~\ref{fig:config_matchpairs_1}.
  Consider the vertex labeling depicted in Figure~\ref{fig:double_aug_initial}. 
Next, we describe the aforementioned operation with input the pair $(s_u,s_v)$.

\noindent
\emph{Double augmentation on $(s_u,s_v)$:} (see Figure~\ref{fig:double_augmentation}) 
  [step 0:] addition of the two new edges 
  $u_1v_0$ and $u_2v_3$. [step 1:] expansion of $v_0$  to $v^1_0vv^2_0$
  in such a way that the partition associated with $v^2_0$ is $\{u_1,u_3\}$ 
  and with $v^1_0$ is $\{v_1,v_3,z_1\}$ and addition of the new edge $vv_2$. 
  [step 2:] addition of the new edge
  $vu_2$. [step 3:] expansion of $u_2$ to $u^1_2uu^2_2$ in such a way that
  the partition associated with $u^1_2$ is $\{u_1, x_2, u_3\}$
  and with $u^2_2$ is $\{v_1, v, v_3\}$ and addition of the new edge
  $uv^2_0$. [step 4:] addition of the new edge $uu_0$.
  We observe that in steps 1 and 3 respectively, expansion of $v_0$ and expansion of $u_2$ respectively are allowed given 
  that the degrees are $5$ and $6$ respectively; recall that degree at least 4 is required for expansion; 
  see Subsection~\ref{subsec:braces}.
  
In case that the supporting edges of \textbf{\emph{uv}} for the pair $(s_u,s_v)$
are configured as in Figure~\ref{fig:config_matchpairs_2} or as in
  Figure~\ref{fig:config_matchpairs_3} respectively (set the same vertex labeling), 
  if we replace the edges added at the step 0 of the double augmentation described above by $u_2v_3, u_3v_0$ and 
  $u_2v_1, u_3v_0 $ respectively,
  then the local configuration obtained is the one depicted in Figure~\ref{fig:double_aug_0}. 
Therefore, if we continue applying steps 1, 2, 3 and 4 as before we obtain the local configuration depicted in 
 Figure~\ref{fig:double_aug_4}. 

\begin{figure}[h]
  \centering 
   \subfigure[Initial configuration]
  { 
    \ifpdf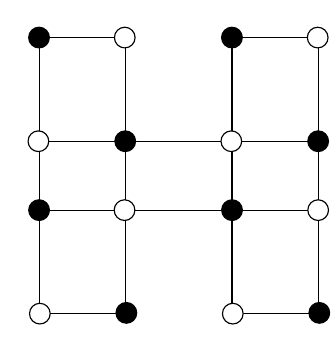\else\input{pair_rect_0.ps_t}\fi
    \label{fig:double_aug_initial}
  }\qquad
    \subfigure[step 0]
  { 
    \ifpdf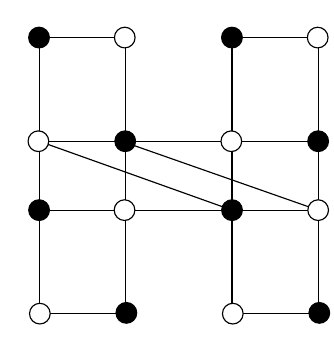\else\input{pair_rect.ps_t}\fi
    \label{fig:double_aug_0}
  }\qquad 
    \subfigure[step 1]
  { 
    \ifpdf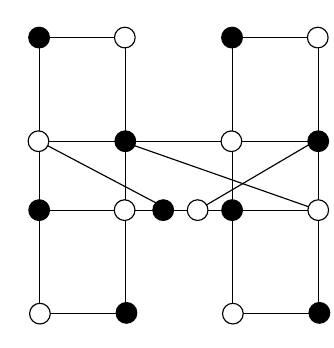\else\input{double_aug.ps_t}\fi
    \label{fig:double_aug_1}
  } 
  
    \subfigure[step 2]
  { 
    \ifpdf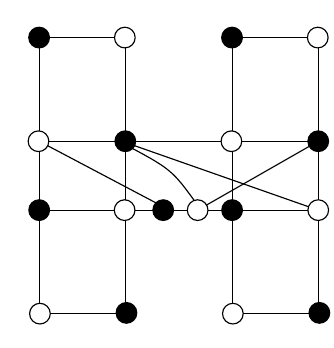\else\input{double_aug_2.ps_t}\fi
    \label{fig:double_aug_2}
  } \qquad
    \subfigure[step 3]
  { 
    \ifpdf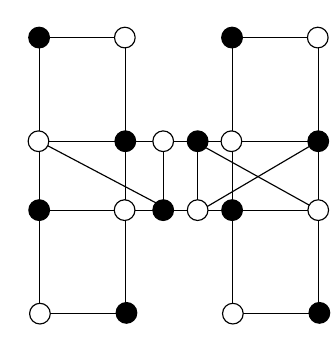\else\input{double_aug_3.ps_t}\fi
    \label{fig:double_aug_3}
  } \qquad
     \subfigure[step 4]
  { 
    \ifpdf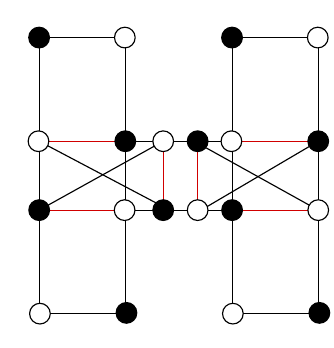\else\input{double_aug_4.ps_t}\fi
    \label{fig:double_aug_4}
  } 
  \caption{Double Augmentation on $(s_u,s_v)$. In subfigure~(f), red edges are depicted by red lines.}
  \label{fig:double_augmentation}
  \end{figure}

We claim that the graph obtained from $Q$ by performing a double augmentation on every pair of 
  matched squares is a hexagon graph of $G$.
The disjoint subsets of vertices $\{u_0,u_3,u_1,u^1_2,u,v^2_0\}$ 
  and $\{v_1,v_2,v_3,v^1_0,v,u^2_2\}$ induce hexagons. Let $h_{u}$ and $h_{v}$ denote them respectively.
The claim follows by setting $\{u_0u_3,u_1u^1_2,uv^2_0\}$ and $\{v_1v_2,v_3v^1_0,vu^2_2\}$ 
 to be the subsets of red edges in $h_{u}$ 
  and $h_{v}$, respectively (see Figure~\ref{fig:double_aug_4}).  
  
To conclude, since steps 0, 2 and 4 correspond to type-1 augmentations, and steps 1 and 3 correspond to
  type-2 augmentations, we have that a double augmentation on a pair of matching related squares
  is composed of a sequence of simple augmentations.
\end{proof}

\section{Construction of square graphs} \label{sec:construction}

In order to prove Theorem~\ref{theo:main}, by Lemma~\ref{theo:transition_braces} it suffices to show that 
  we can construct an $M$-square graph of $G$, for some perfect matching $M$ of $G$, from the ladder
  on $8$ vertices using simple augmentations. 
In this section we develop a method to construct square graphs 
  following an ear decomposition of the underlying cubic bridgeless graph 
  $G$ and using simple augmentations.

\subsection{Odd ear decomposition of a cubic bridgeless graph} \label{eardecomp}

Let $G$ be a graph. We say that a path, or a cycle of $G$, 
  is \emph{even} (\emph{odd} respectively) if it has an even (odd respectively) number of edges.
An \emph{odd ear decomposition} of $G$, denoted by $(G_0, G_i, P_i)^l$, 
  consists of a sequence of subgraphs $G_0, G_1, \ldots, G_l$ and a sequence of odd paths $P_1, \ldots, P_l$
  of $G$ such that $G_0$ is an even cycle of~$G$, $G_l = G$ and for each $i \in \{1, \ldots,l\}$ 
  the subgraph $G_{i}$ is obtained from $G_{i-1}$ joining two vertices $\alpha_i$ and $\beta_i$ in $V(G_{i-1})$ 
  by a path $P_i$, where $P_i$ is such that $V(P_i) \cap V(G_{i-1}) = \{\alpha_i, \beta_i\}$ 
   and $E(P_i) \cap E(G_{i-1})=\emptyset$. 
It is folklore that every edge of a cubic bridgeless graph is contained in a perfect
matching and hence, the class of cubic bridgeless graph is a subclass of the class of 1-extendable graphs.
In addition, every 1-extendable graph admits an odd ear decomposition~\cite[\S 5.4]{B:lovasz-plummer}.  

Let $G$ be a cubic bridgeless graph and $(G_0, G_i, P_i)^l$ be an odd ear decomposition of $G$.
 We say that a perfect matching $M$ of $G$ is \emph{absolute} in $(G_0, G_i, P_i)^l$ if the restriction
 of $M$ to $E(G_i)$ is a perfect matching of $G_i$ for every $i\in\{0,1,\ldots,l\}$. The next observation is 
 straightforward.
 
 \begin{observation}
For every odd ear decomposition $(G_0, G_i, P_i)^l$ of a cubic bridgeless graph $G$,
there exists a perfect matching $M$ of $G$ that is absolute in $(G_0, G_i, P_i)^l$.
 \end{observation}

In the rest of the paper, we deal only with perfect matchings that are absolute in a given odd ear decomposition
$(G_0, G_i, P_i)^l$.
Let $i \in \{1,\ldots,l\}$ and let $V_j(G_i)$ denote the subset of vertices of $V(G_i)$ 
  that have degree $j$ in $G_i$ for each $j \in \{2,3\}$. 
  Let $u,v \in V_3(G_i)$ and $P$ be a path of $G_i$
  with end vertices $u, v$ such that $V(P) \cap V_3(G_i) = \{u,v\}$. In other words,
  every inner vertex of $P$ belongs to $V_2(G_i)$. 
  We say that $P$ is a $(u,v)$-path of $G_i$ and usually denote $P$ by $p(u,v)$. 
  Note that there may exist multiple $(u,v)$-paths.
  We shall denote by $\mathcal{P}(G_i)$ the set of all $(u,v)$-paths for all $u, v$ in $V_3(G_i)$.

We note that if $v$ is a vertex in $V_3(G_i)$, 
	    then there are three (not necessarily distinct) vertices $x,y,z$ in $\mathcal{V}(G_i)$,  
	    such that $p(x,v),$ $p(y,v),$ $p(z,v) \in \mathcal{P}(G_i)$.
	    We say that the set $\{x, y, z\}$ is the \emph{set of pseudo-neighbors} of $v$ in $G_i$.
	    
Observe that if $M$ is a perfect matching of $G$ and $vw \in M$, then there is a 
	    unique path $P \in \{p(x,v),\, p(y,v),\, p(z,v)\}$ such that $vw \in E(P)$.
	    We refer to $P$ as the \emph{matching-path} of $v$ in $G_i$ (with respect to $M$).
	    If $vw$ is not in $E(P)$, then $P$ is called a \emph{cycle-path} of $v$ in $G_i$.
	  Note that a path $p(u,v)$ in $\mathcal{P}(G_i)$ could be both, a matching-path of $v$ and a cycle-path of $u$.
However, since $M$ is a perfect matching that is absolute in $(G_0, G_i, P_i)^l$, the path
$P_i = p(\alpha_i, \beta_i) \in \mathcal{P}(G_i)$ is always a cycle-path of both $\alpha_i$ and $\beta_i$ in $G_i$
 (see Figure~\ref{fig:pro_paths_1}).

In Subsection~\ref{subsec:ear-square_graphs}, we generalize the definition of square graphs of a cubic
      graph $G$ to the intermediate graphs $G_0, G_1, \ldots, G_l$ associated with an odd ear decomposition
      of $G$.

\subsection{Ear square graphs}\label{subsec:ear-square_graphs} 

In this section and in the rest of the paper, $G$ is a cubic bridgeless graph, $(G_0, G_i, P_i)^l$ is an odd ear decomposition of $G$ 
 and $M$ is a perfect matching of $G$ that is absolute in $(G_0, G_i, P_i)^l$.

 \begin{definition}[Ear square graphs]\label{def:ear_square}
For each $i \in \{1,\ldots,l\}$, a \emph{($G_i,M)$-ear square graph} is a bipartite graph $Q_i$ 
with neither loops nor multiple edges that 
satisfies the following properties:  
\begin{enumerate}
\item For each vertex $v$ in $V_3(G_i)$, 
      the graph $Q_i$ has a square $s_v$. For every $u$,$v$ in $V_3(G_i)$
      with $u \neq v$,  the squares $s_v$ and $s_u$ are vertex disjoint subgraphs of $Q_i$. 
      Moreover, $V(Q_i) = \{V(s_v): v \in V_3(G_i)\}$.
\item The set of edges of $Q_i$ is given by $$\{E(s_v): v \in V_3(G_i)\} \dot{\bigcup_{p(u,v)\in \mathcal{P}(G_i)}} 
      {{\bf{p}}(u,v)} $$ 
      where $\{\pp(u,v): p(u,v)\in \mathcal{P}(G_i)\}$ is defined such that the following conditions hold:
      \begin{enumerate}
      \item For each $p(u,v)\in \mathcal{P}(G_i)$, we have $|\pp(u,v)|=2$, and there are edges $e_u$ in $E(s_{u})$, 
	    $e_v$ in $E(s_{v})$ such that the subgraph of $Q_i$ induced
	    by the set of edges $\{e_u,e_v\} \cup \pp(u,v)$ is isomorphic to $K_{2,2}$.
	    The edges $e_u$ and $e_v$ are called \emph{the supporting edges of $\pp(u,v)$}
	    in $s_u$ and $s_v$, respectively.
     \item Let $v$ be a vertex in $V_3(G_i)$ and $\{x,y,z\}$ be its set of pseudo-neighbors.  
	    If $p(x,v)$ is the matching-path of $v$ in $G_i$,
	    then the supporting edges of $\pp(v,y)$ and $\pp(v,z)$ in $s_v$ 
	    are vertex disjoint (see Figure~\ref{fig:ear_square_v}).
     \item Elements in $\{\pp(u,v): p(u,v)\in \mathcal{P}(G_i)\}$ are pairwise disjoint.
    \end{enumerate}
\end{enumerate}
\end{definition}

\begin{figure}[h]
  \centering 
    \subfigure[]
  { 
    \ifpdf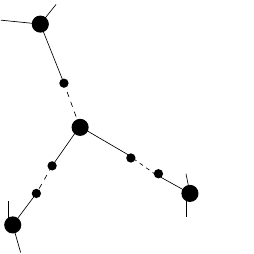\else\input{projected_paths.ps_t}\fi
    \label{fig:proj_0}
  } 
    \subfigure[]
  { 
    \ifpdf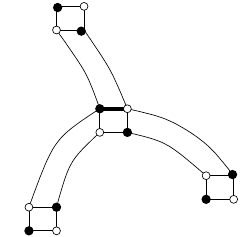\else\input{proyected_paths_1.ps_t}\fi
    \label{fig:proj_1}
  } 
      \subfigure[]
  { 
    \ifpdf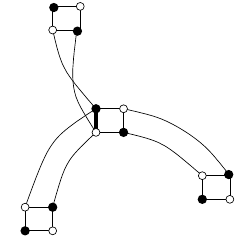\else\input{proyected_paths_2.ps_t}\fi
    \label{fig:proj_2}
  } 
      \subfigure[]
  { 
    \ifpdf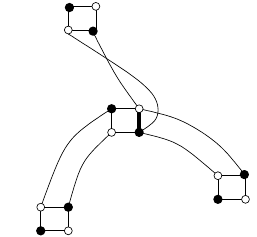\else\input{proyected_paths_3.ps_t}\fi
    \label{fig:proj_3}
  } 
      \subfigure[]
  { 
    \ifpdf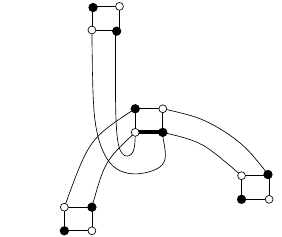\else\input{proyected_paths_4.ps_t}\fi
    \label{fig:proj_4}
  } 
  \caption{Local representation of $G_i$ and a $(G_i,M)$-ear square graph of $G_i$. 
  In subfigure~(a), we depict a vertex $v \in V_3(G_i)$ with $x, y, z \in V_3(G_i)$ its pseudo-neighbors
  and $p(v,x)$ the matching-path of $v$ in $G_i$. Dashed edges represent edges from $M$.
  In subfigures (b)-(e), we depict all the allowed locations of the 
  supporting edge of $\pp(v,x)$ in $s_v$ in a $(G_i,M)$-ear square graph of $G_i$.
  In each subfigure the supporting edge is depicted by a thicker line.}
  \label{fig:ear_square_v}
  \end{figure}


  For every $p(u,v) \in \mathcal{P}(G_i)$, the set $\pp(u,v)$ 
  is said to be its \emph{projected $(u,v)$-path in $Q_i$}.
If $p(u,v)$ is the matching-path of $v$ in $G_i$, we say 
    that $\pp(u,v)$ is the \emph{projected matching-path of $s_{v}$ in $Q_i$}.

Since $V_3(G_l)= V(G)$, the following proposition follows from  Definition~\ref{def:square} and Definition~\ref{def:ear_square}.

\begin{observation} \label{prop:ear_squares} 
A graph $H$ is a $(G_l,M)$-ear square graph if and only if $H$ is an $M$-square graph.
\end{observation}

In Lemma~\ref{theo:main_theorem} we formalize the construction 
  of square graphs using ear square graphs and simple augmentations.  
This lemma is proved in Section~\ref{sec:proofsecondtheorem}.

\begin{lemma}[Construction of square graphs]\label{theo:main_theorem}
  Let $G$ be a cubic bridgeless graph, $(G_0, G_i, P_i)^l$ be an odd ear decomposition of $G$ 
 and $M$ be a perfect matching of $G$ that is absolute in $(G_0, G_i, P_i)^l$.
  Let $L_{8}$ denote the ladder on $8$ vertices (see Figure~\ref{fig:ladders}).
  The following two properties hold.
  \begin{enumerate}
    \item \label{theo:main_theorem_item1}  A $(G_1,M)$-ear square graph $Q_1$ can be generated from 
      $L_8$ using type-1 augmentations.
    \item \label{theo:main_theorem_item2} Let $i\in\{2,\ldots,l\}$ and $Q_{i-1}$ be a $(G_{i-1},M)$-ear square graph.
    Then a $(G_i,M)$-ear square graph $Q_i$ can be generated from $Q_{i-1}$ using a sequence of simple augmentations.  
   \end{enumerate}
\end{lemma}

Note that Lemma \ref{theo:main_theorem} along with Observation~\ref{prop:ear_squares} and Lemma~\ref{theo:transition_braces}  
imply Theorem~\ref{theo:main}.

\section{Proof of Lemma~\ref{theo:main_theorem}}\label{sec:proofsecondtheorem}

In this section, $G$ is a cubic bridgeless graph, 
$(G_0, G_i, P_i)^l$ is an odd ear decomposition of $G$ 
 and $M$ is a perfect matching of $G$ that is absolute in $(G_0, G_i, P_i)^l$.
  Let $L_{8}$ denote the ladder on $8$ vertices. 
Moreover, for each $i\in \{1,\ldots, l\}$, let $Q_i$ denote a $(G_i,M)$-ear square graph.

In what follows we enunciate two natural properties about ear square graphs. 
The result of Proposition~\ref{lemma:degrees} follows directly 
  from Definition~\ref{def:ear_square}.
  
\begin{proposition}\label{lemma:degrees} For every $i\in \{1,\ldots, l\}$, 
  each square $s_{v}$ in $Q_i$ with $V(s_{v})= \{v_j: j \in \mathbb{Z}_{4}\}$ 
  is such that there exists a unique $j \in \mathbb{Z}_{4}$ such that $d_{v_{j}} = d_{v_{j+1}} = 4$
  and $d_{v_{j+2}} = d_{v_{j+3}} = 3$.
\end{proposition}

\begin{proposition}\label{lemma:intersection_of_pp} Let $p(x,y)$ and $p(w,z)$ be paths
in $\mathcal{P}(G_i)$. Let $\pp(x,y)$ be the projected path of $p(x,y)$
and $\pp(w,z)$ be the projected path of $p(w,z)$ in $Q_i$. 
Then the subgraph $S$ of $Q_i$ with a set of 
  edges $ \pp(x,y) \cup \pp(w,z) \cup E(s_x \cup s_y \cup s_w \cup s_z)$
  is isomorphic to one of the 9 graphs (configurations) depicted in Figure~\ref{fig:intersection_of_pp}.
\end{proposition}
\begin{figure}[h]
  \centering
    \subfigure[Configuration 1]
{ 
  \ifpdf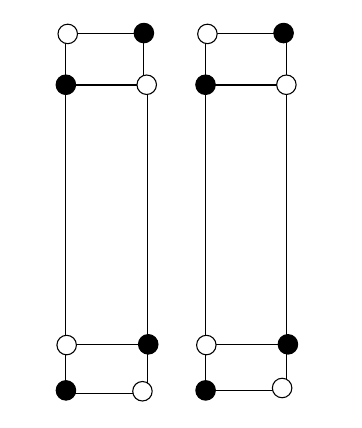\else\input{case1.ps_t}\fi
  \label{case1}
} 
\subfigure[Configuration 2]
{
  \ifpdf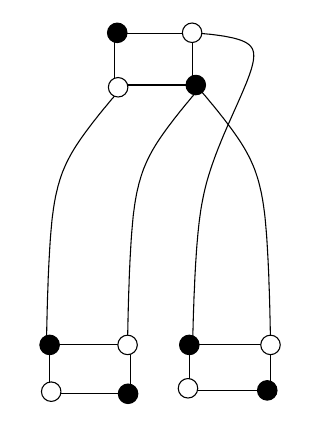\else\input{case3.ps_t}\fi
  \label{case2}
}
\subfigure[Configuration 3]
{
  \ifpdf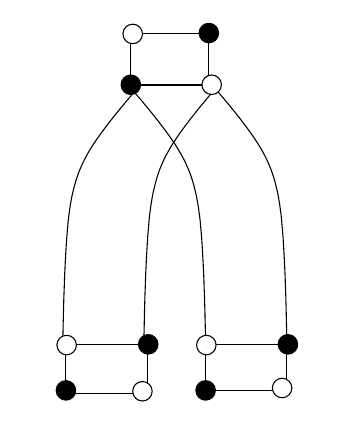\else\input{case2.ps_t}\fi
  \label{case3}
}
\subfigure[Configuration 4]
{
  \ifpdf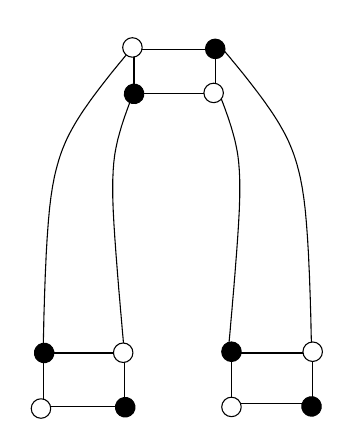\else\input{case3_1.ps_t}\fi
  \label{case4}
}
\subfigure[Configuration 5]
{
  \ifpdf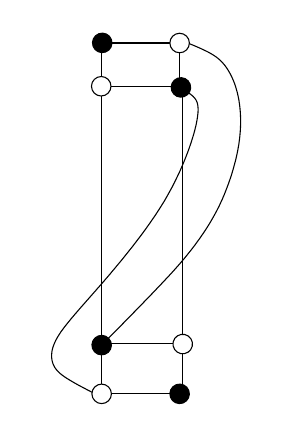\else\input{case4.ps_t}\fi
  \label{case5}}
    \subfigure[Configuration 6]
{
  \ifpdf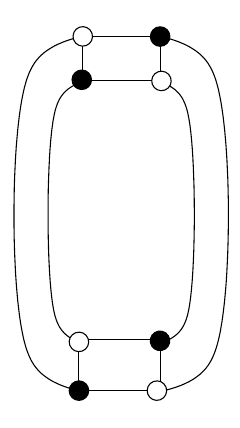\else\input{case4_1.ps_t}\fi
  \label{case6}
}
\subfigure[Configuration 7]
{
  \ifpdf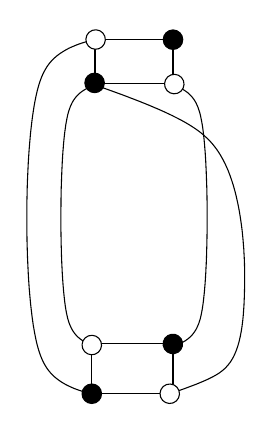\else\input{case4_2.ps_t}\fi
  \label{case7}
}
\subfigure[Configuration 8]
{
  \ifpdf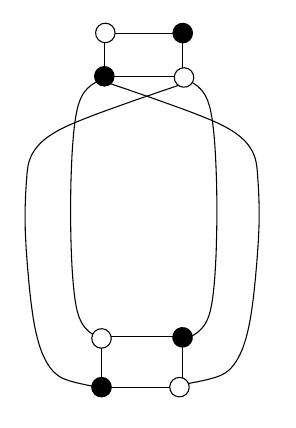\else\input{case4_3.ps_t}\fi
  \label{case8}
}
\subfigure[Configuration 9]
{
  \ifpdf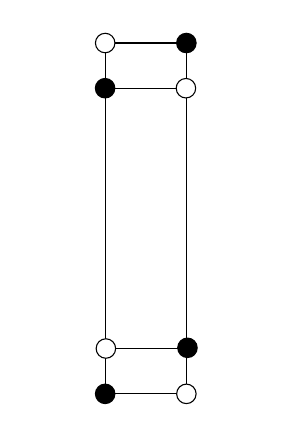\else\input{case5.ps_t}\fi
  \label{case9}
}
\caption{In (a) is depicted the unique 
subgraph that arises when vertices $x, y, z, w \in V_3(G_i)$ are all distinct.
From (b) to (d) the three possible subgraphs that arise when 
$|\{s_x,s_y,s_w,s_z\}|=3$. Figures from (e) to (i) depict all the possible 	situations when $|\{s_x,s_y,s_w,s_z\}|=2$.}
    \label{fig:intersection_of_pp}
\end{figure}

\begin{proof}
We first suppose that $|\{x,y,w,z\}|=4$. Then $x, y, z, w \in V_3(G_i)$ are all distinct
and the squares $s_x$, $s_y$, $s_w$, $s_z$ in $Q_i$ are vertex disjoint.
Therefore, $S$ is isomorphic to the graph depicted in Figure~\ref{case1}. 

We now suppose that $|\{x,y,w,z\}|=3$. It means that the paths $p(x,y)$ and $p(w,y)$ have one
common end vertex. Without loss of generality we suppose that $x=w$, and then, 
$s_x$, $s_y$ and $s_z$ are vertex disjoint. 
In the subgraph $S$ three distinct situations depending on the location
of the supporting edges $e_{x}$ and $e_{w}$ of $\pp(x,y)$ and $\pp(w,z)$ in $s_x$ can arise:
\begin{enumerate}
 \item[a.1)] either $|e_{x} \cap e_{w}| =1$, or
 \item[a.2)] $e_{x} = e_{w}$, or
 \item[a.3)] $e_{x} \cap e_{w}= \emptyset$.
\end{enumerate}
If situation a.1) holds, then 
$S$ is isomorphic to configuration~2, see Figure~\ref{case2}.   
If situation a.2) holds, then 
$S$ is isomorphic to configuration~3, see Figure~\ref{case3},   
and if situation a.3) holds, then 
$S$ is isomorphic to configuration~4, see Figure~\ref{case4}.

We finally suppose that $|\{x,y,w,z\}|=2$. Without loss of generality we assume that $x=w$ and $y=z$.
If $p(x,y) = p(w,z)$, then $S$ is isomorphic to the graph depicted in~Figure~\ref{case9},
this graph is called configuration~9.
Otherwise, in the graph $S$ several distinct situations depending on the location of the supporting
edges $e_x$ and $e_w$ of $\pp(x,y)$ and $\pp(w,z)$ in $s_x$ 
and of the supporting edges $e_y$ and $e_z$ of $\pp(x,y)$ and $\pp(w,z)$ in $s_y$
may arise:

\begin{enumerate}
 \item[b.1)] either $|e_x \cap e_w| =1$ and $|e_y \cap e_z| =1$, or
 \item[b.2)] $e_x \cap e_w = \emptyset$ and $e_y \cap e_z = \emptyset$, or
 \item[b.3)] $|e_x \cap e_w| =1$ and $e_y \cap e_z = \emptyset$, or
  \item[b.4)]$e_x = e_w$ and $e_y \cap e_z =\emptyset$, or 
 \item[b.5)] $|e_x \cap e_w| =1$ and $e_y = e_z$, or $e_x = e_w$ and $e_y = e_z$.
\end{enumerate}

If situation b.1), b.2), b.3) or b.4) holds, then $S$ is isomorphic
 to configuration~5, 6, 7, or 8, respectively. Those configurations are depicted 
 in Figure~\ref{fig:intersection_of_pp}).
Situations described in b.5) do not occur in $Q_i$ given that $Q_i$ 
does not have multiple edges. We clarify the last statement 
in the following paragraph.

If we suppose that $|e_x \cap e_w| =1$ or $e_x = e_w$ 
then, there exists a vertex, without loss of generality we assume that such a vertex is $x_1 \in e_x \cap e_y$
such that $x_1y_1$ and $x_1z_1$ are edges of $S$. 
We recall that $Q_i$ does not have multiple edges. Since $e_y = e_z$, we have $y_1=z_1$ and $S$ has a double edge, 
a contradiction.
\end{proof}

\subsection{Generating ear square graphs}

This section is devoted to prove Lemma~\ref{theo:main_theorem}.

\subsubsection*{Proof of Lemma~\ref{theo:main_theorem}, part \textit{\ref{theo:main_theorem_item1}}}

We need to prove that we can generate a $(G_1,M)$-ear square graph from $L_8$ using type-1 augmentations (addition of new edges).
We consider the vertex-labeling of $L_8$ depicted in Figure~\ref{fig:start_squares_ladder}.
Let $p(u,v)$, $p(x,y)$ and $p(w,z)$ be the only three paths in $\mathcal{P}(G_1)$,
where $u=x=w$, $v=y=z$ and $P_1=p(u,v)$. We have that $G_1$ satisfies one of the following properties: 

\begin{itemize}
 \item[1)] either $p(x,y)$ is the matching-path of $v$ and of $u$ in $G_1$, or
 \item[1$'$)] $p(w,z)$ is the matching-path of $v$ and of $u$ in $G_1$, or
 \item[2)] $p(x,y)$ is the matching-path of $v$ in $G_1$ and $p(w,z)$ is the matching-path of $u$ in $G_1$, or
  \item[2$'$)] $p(w,z)$ is the matching-path of $v$ in $G_1$ and $p(x,y)$ is the matching-path of $u$ in $G_1$.
\end{itemize}

By symmetry, it suffices to prove that for each $i\in\{1,2\}$, we can generate from $L_8$  a 
  $(G_1,M)$-ear square graph,
  where $G_1$ and $M$ satisfies i). We first claim that if 1) holds, then the bipartite graph obtained 
  from $L_8$ by adding the new edges $v_0u_2$ and $v_3u_1$ is a $(G_1,M)$-ear square graph 
  (see Figure~\ref{fig:start_squares}). The validity of this claims follows from considering 
  $\{v_0u_2, v_3u_1\}$ to be the projected path of $p(x,y)$,
  $\{v_2u_2, v_3u_3\}$ to be the projected path of $p(w,z)$ and 
  $\{v_0u_0,v_1u_1\}$ to be the projected path of $p(u,v)$.

Secondly, we claim that if 2) holds, then the bipartite graph obtained 
  from $L_8$ by adding the new edges $v_1u_3$ and $v_0u_2$ is a $(G_1,M)$-ear square graph
  (see Figure~\ref{fig:start_squares_c}).
In this case, if we let $\{v_0u_2, v_3u_1\}$ be the projected path of $p(u,v)$,
  $\{v_2u_2, v_3u_3\}$ be the projected path of $p(w,z)$ and 
  $\{v_0u_0,v_1u_1\}$ be the projected path of $p(x,y)$, then the claim follows.

\begin{figure}[h] 
 \centering 
     \subfigure[A ladder $L_8$ on $8$ vertices]
{ 
  \ifpdf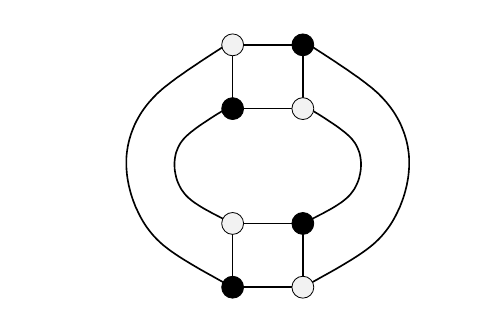\else\input{starting_rectangles_ladder.ps_t}\fi
  \label{fig:start_squares_ladder}
} 
    \subfigure[$G_1$-ear square graph]
{ 
  \ifpdf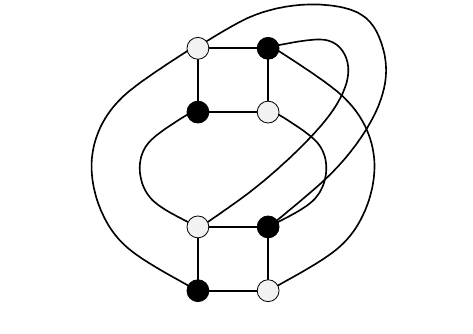\else\input{starting_rectangles.ps_t}\fi
  \label{fig:start_squares}
} 
    \subfigure[$G_1$-ear square graph]
{ 
  \ifpdf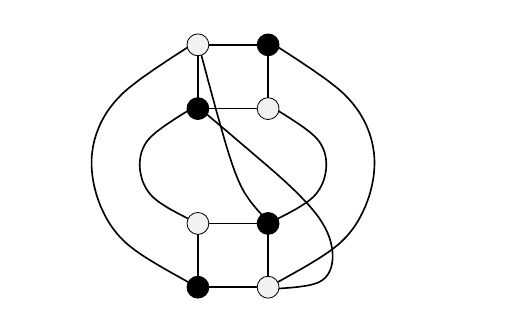\else\input{starting_rectangles_1.ps_t}\fi
  \label{fig:start_squares_c}
} 
 \caption{Generation of $Q_1$ from $L_8$.}
\label{fig:start}
\end{figure}

\subsubsection*{Proof of Lemma~\ref{theo:main_theorem}, part \textit{\ref{theo:main_theorem_item2}}}

For each $i\in \{2,\ldots,l\}$, we need to show that from a 
  $(G_{i-1},M)$-ear square graph we can generate a  
  $(G_{i},M)$-ear square graph using simple augmentations.
For this purpose, the idea is to make local changes; 
we basically replace the projected paths in $Q_{i-1}$
of the paths that contain $\alpha_i$ and $\beta_i$ 
by two new squares $s_{\alpha_i}$, $s_{\beta_i}$, and by the new projected paths incident with them.
Moreover, we modify neither any square in $Q_{i-1}$, nor the position of the supporting edges
of the projected paths incident with them (see Figure~\ref{fig:pro_paths}).
Here, $\alpha_i$, $\beta_i$ denote the end vertices of the path $P_i$ from $(G_0, G_i, P_i)^l$.

Let $\pp(x,y)$ and $\pp(w,z)$ be the projected paths in $Q_{i-1}$ 
  such that $\alpha_i$ belongs to $V(p(x,y))$ in $G_{i-1}$ and 
  $\beta_i$ belongs to $V(p(w,z))$ in $G_{i-1}$.

\begin{figure}[h] 
 \centering 
  \subfigure[$P_i=:\alpha_i \cdots \beta_i$ is a
  cycle-path of $\alpha_i$ and $\beta_i$ in $G_i$. Paths 
  $p(x,y)$ and $p(w,z)$ contain $\alpha_i$ and $\beta_i$ in $G_{i-1}$.]
{ 
  \ifpdf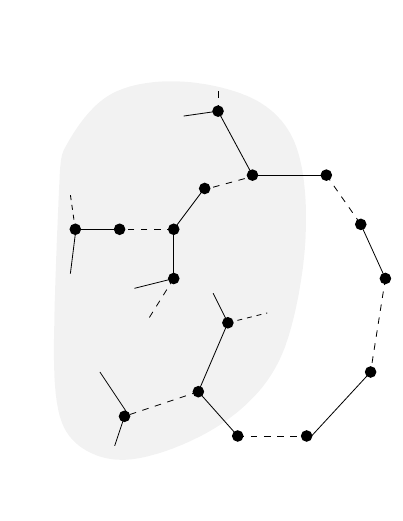\else\input{projected_paths_1.ps_t}\fi
  \label{fig:pro_paths_1}
} \qquad
 \subfigure[$\pp(x,y)$ and $\pp(w,z)$ are the projected paths 
 of $p(x,y)$ and $p(w,z)$ in $Q_{i-1}$.]
{ 
  \ifpdf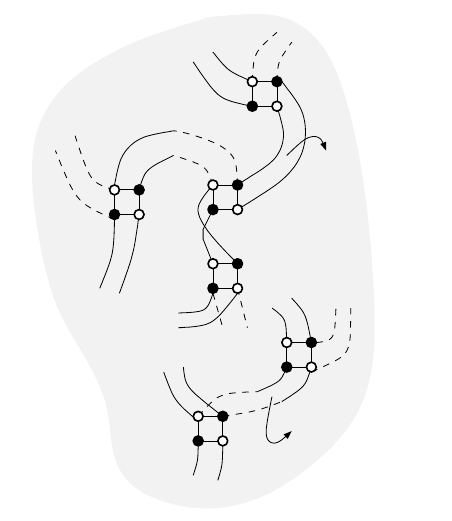\else\input{projected_paths_2.ps_t}\fi
  \label{fig:pro_paths_2}
} \qquad
 \subfigure[$(G_i,M)$-ear square graph. Squares $s_{\alpha_i}$ and $s_{\beta_i}$
 are constructed and also the projected paths incident with them.]
{ 
  \ifpdf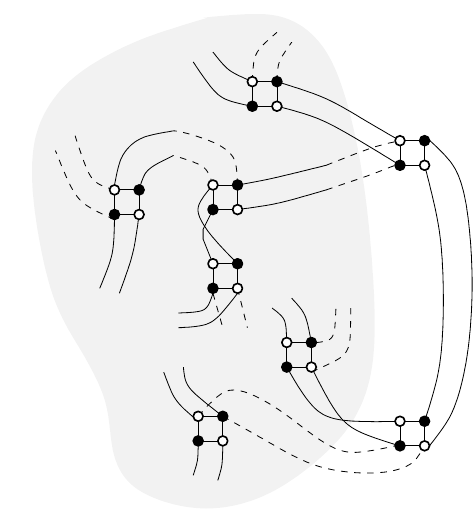\else\input{p-p-3.ps_t}\fi
  \label{fig:pro_paths_3}
} 
 \caption{$\pp(x,y)$ and $\pp(w,z)$ are the projected paths in $Q_{i-1}$ 
  such that $\alpha_i \in V(p(x,y))$ and 
  $\beta_i \in V(p(w,z))$ in $G_{i-1}$. In (a), dashed edges represent the perfect matching
  $M$ that is absolute in $(G_0, G_i, P_i)^l$. In (b)-(c), dashed lines represent projected matching-paths.}
 \label{fig:pro_paths}
\end{figure}

In what follows, for the sake of simplicity we set $u=\alpha_i$ and $v=\beta_i$.  
We attempt to generate the two new squares $s_{u}$ and $s_{v}$ and the projected paths 
$\pp(u,x), \pp(u,y), \pp(u,v), \pp(v,z)$ and $\pp(v,w)$.
In order to cover all cases we need to take care of two issues, first the interaction 
between the projected paths $\pp(x,y)$ and $\pp(w,z)$ in $Q_{i-1}$, which is described
by Proposition~\ref{lemma:intersection_of_pp} and depicted in Figure~\ref{fig:intersection_of_pp},
and the second issue is the location of the perfect matching $M$ with respect to the edges
and paths incident with $u$ and $v$. For the second issue, we know that $P_i$ is a cycle-path (
recall that since $M$ is a perfect matching that is absolute in $(G_0, G_i, P_i)^l$, the path
$P_i = p(\alpha_i, \beta_i) \in \mathcal{P}(G_i)$ is always a cycle-path of both $u$ and $v$ in $G_i$ --- see Figure~\ref{fig:pro_paths_1}), and therefore
if $p(x,y) \neq p(w,z)$, then the matching-path of $u$ is either $p(u,y)$ or $p(u,x)$ and 
the matching-path of $v$ is either $p(v,w)$ or $p(v,z)$. In Figure~\ref{fig:cases_matching-related_1}, 
we describe the cases and depict examples for the situation that $x=w$ and $y \neq z$.
The remaining cases, for example when all $x,w,y,z$ are different, are analogous.

\begin{figure}[h]
 \centering 
  \subfigure[instance i): $p(u,y)$ matching-path of $u$ and
  $p(v,z)$ matching-path of $v$.]
{ 
  \ifpdf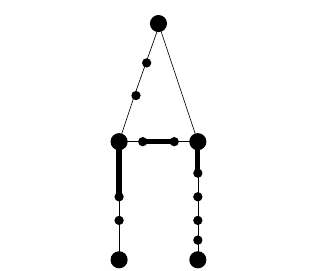\else\input{case1_matchrelated_c.ps_t}\fi
  \label{fig:match_related_1}
} \quad
  \subfigure[instance ii): $p(u,x)$ matching-path of $u$ and
  $p(v,z)$ matching-path of $v$.]
{ 
  \ifpdf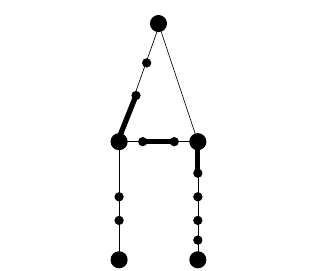\else\input{case2_matchrelated_c.ps_t}\fi
  \label{fig:match_related_2}
} \quad
  \subfigure[instance iii): $p(u,y)$ matching-path of $u$ and
  $p(v,w)$ matching-path of $v$.]
{ 
  \ifpdf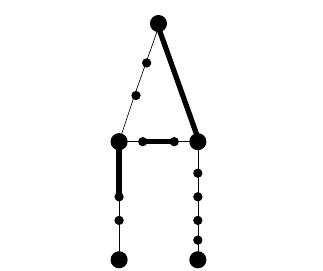\else\input{case3_matchrelated_c.ps_t}\fi
  \label{fig:match_related_3}
}  \quad
  \subfigure[instance iv): $p(u,x)$ matching-path of $u$ and
  $p(v,w)$ matching-path of $v$.]
{ 
  \ifpdf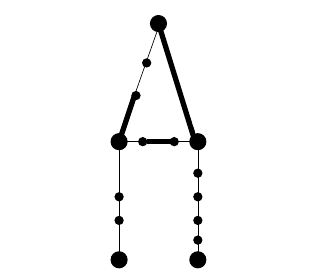\else\input{case4_matchrelated_c.ps_t}\fi
  \label{fig:match_related_4}
} 
 \caption{In the subfigures we depicted an example of each instance when the paths $p(x,y)$ and $p(w,z)$
	  in $G_{i-1}$ intersect in one vertex (see Configurations 2,3 and 4 in Figure~\ref{fig:intersection_of_pp}). 
          Bold edges represent the perfect matching. Moreover, $p(u,v)=P_i$.}
 \label{fig:cases_matching-related_1}
\end{figure}

In the case that $p(x,y)=p(w,z)$, without loss of generality we can assume that $u$ and $v$ are
placed (with respect to $x$ and $y$) as depicted in Figure~\ref{fig:cases_config9}. Then, we have that the matching-path of $u$ 
is either $p(u,x)$ or $p(u,v)$ (with $p(u,v) \neq P_i$) and 
the matching-path of $v$ is either $p(v,y)$ or $p(u,v)$ (with $p(u,v) \neq P_i$). 
In Figure~\ref{fig:cases_config9}, we describe these situations with a corresponding example.

\begin{figure}[h]
 \centering 
  \subfigure[instance i$'$): $p(u,x)$~matching-path of $u$ and
  $p(v,y)$ matching-path of $v$.]
{ 
  \ifpdf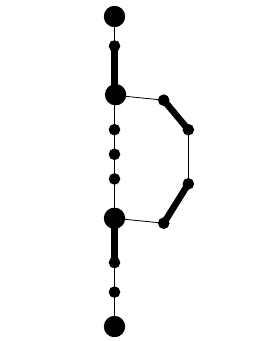\else\input{config9_i.ps_t}\fi
  \label{fig:match_related_1}
} \qquad 
  \subfigure[instance ii$'$): $p(u,x)$ matching-path of $u$ and
  $p(u,v)$ matching-path of $v$.
  ]
{ 
  \ifpdf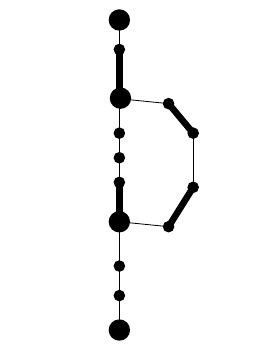\else\input{config9_ii.ps_t}\fi
  \label{fig:match_related_2}
} \qquad 
  \subfigure[instance iii$'$): $p(u,y)$ matching-path of $u$ and
  $p(v,w)$ matching-path of $v$. ]
{ 
  \ifpdf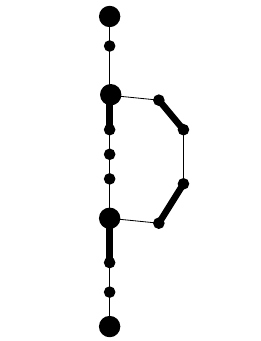\else\input{config9_iii.ps_t}\fi
  \label{fig:match_related_3}
}  \qquad
  \subfigure[instance iv$'$): $p(u,v)$ matching-path of $u$ and $v$.]
{ 
  \ifpdf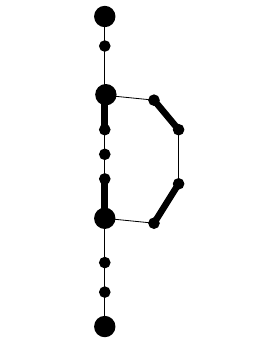\else\input{config9_iv.ps_t}\fi
  \label{fig:match_related_4}
} 
 \caption{In each figure is depicted an example of the distinct instances in the case that $p(x,y) = p(w,z)$ (see configuration 9 
	  in Figure~\ref{fig:intersection_of_pp}). Bold edges represent the perfect matching. }
 \label{fig:cases_config9}
\end{figure}

Summarizing, to prove Lemma~\ref{theo:main_theorem}.\textit{\ref{theo:main_theorem_item2}},
it suffices to prove that from each configuration of the projected paths $\pp(x,y)$ and $\pp(w,z)$ 
it is possible to generate all instances i), ii), iii) and iv) in case that $\pp(x,y) \neq \pp(w,z)$
and that it is possible to generate all instances i$'$), ii$'$), iii$'$) and iv$'$) in case that $\pp(x,y) = \pp(w,z)$.

Before we go into the analysis of the configurations we shall present an operation 
  consisting of a sequence of simple-augmentations that we constantly use in order to construct
  two new squares; we shall call this operation a \emph{basic square construction}.
This operation is very useful and crucial to reduce the number of cases.

The input of the basic square construction is the bipartite subgraph  
  graph depicted in Figure~\ref{fig:basic_square_0} with distinguished 
  edges $e_a$, $e_b$, $e_c$ and $e_d$. The output 
  is the bipartite subgraph depicted in Figure~\ref{fig:basic_square_2} with distinguished 
  edges $e_a$, $e'_b$, $e'_c$ and $e'_d$. 
In Figure~\ref{fig:basic_square_construction}, 
  we depict the sequence of simple-augmentations that compose the basic square construction.
It is clear that if $H$ is the graph obtained by applying the basic square construction 
  in a subgraph of a brace, then $H$ is a brace. In what follows, we constantly use this operation
  and the latter remark.

\begin{figure}[h]
  \centering 
    \subfigure[]
  { 
    \ifpdf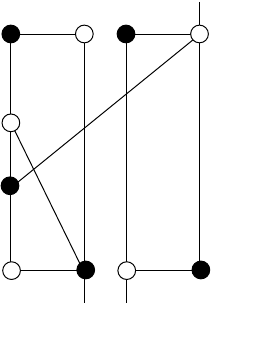\else\input{basic_square_0.ps_t}\fi
    \label{fig:basic_square_0}
  } \hspace*{-1.2cm}
    \subfigure[]
  { 
    \ifpdf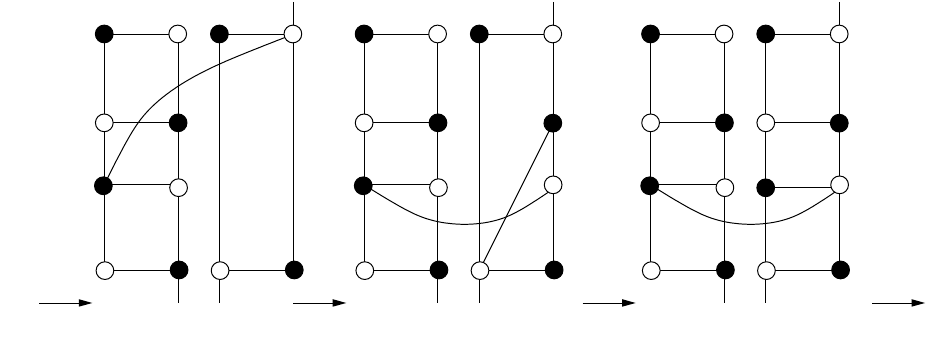\else\input{basic_square_1.ps_t}\fi
    \label{fig:basic_square_1}
  } \hspace*{-1.2cm}
      \subfigure[]
  { 
    \ifpdf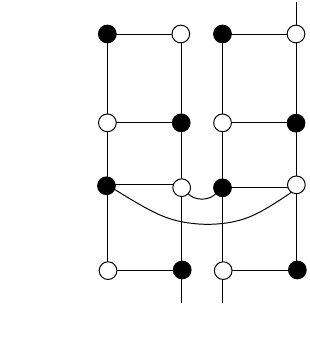\else\input{basic_square_2.ps_t}\fi
    \label{fig:basic_square_2}
  } 
  \caption{Basic square construction. In subfigure (a) the input subgraph
  with distinguished edges $e_a$, $e_b$, $e_c$, $e_d$, in subfigure (b) the steps 
  of the basic square construction and in subfigure (c) the output 
  subgraph with distinguished edges $e_a$, $e'_b$, $e'_c$, $e'_d$,}
  \label{fig:basic_square_construction}
  \end{figure}

Let $V(s_x)= \{x_j: j \in \mathbb{Z}_4\}$, 
  $V(s_y)= \{y_j: j \in \mathbb{Z}_4\}$, $V(s_w)= \{w_j: j \in \mathbb{Z}_4\}$, $V(s_z)= \{z_j: j \in \mathbb{Z}_4\}$.
  Without loss of generality we assume that $\pp(x,y)=\{x_1y_1, x_2y_2\}$, $\pp(w,z)=\{w_1z_1, w_2z_2\}$, and 
  that $x_1$, $w_1,$ $y_2,$ $z_2$ are in the same partition class, depicted in black in Figure~\ref{fig:intersection_of_pp}.
We first study configurations 2, 5, and 7, then 3 and 8, afterwards configurations 4 and 6, and finally configurations 1 and 9.

\begin{description}
   \item[Configurations 2, 5 and 7] 
are respectively depicted in Figu\-res~\ref{case2},~\ref{case5} and~\ref{case7}.
These configurations have a common property, namely: with the notation of Figure~\ref{fig:intersection_of_pp}
each of these configurations may be obtained from Figure~\ref{fig:starting_case} by possible identifying 
$y_2, z_2$ (case of Configuration 5).
Next we show how to generate instances i), ii), iii) and iv) of Figure~\ref{fig:cases_matching-related_1}.

Generation of instances i) and iv): 
		Let $Q^{2}_{i-1}$ be the graph obtained from $Q_{i-1}$ by expanding the vertex $x_1$ 
	        to $x^1_1v_1x^2_1$ in such a way that the partition associated with the vertex $x^2_1$ is 
	        either $\{x_2,z_1\}$ if we are generating instance i), or $\{w_2,z_1\}$ if we are 
	        generating instance iv). Then, we add the new edge  $v_1z_2$ if we are generating instance i),
	        or $v_1y_2$ if we are generating instance iv) ---see Figures~\ref{fig:casei} and~\ref{fig:caseiv} 
	        without the bold edges for an ilustration of $Q^{2}_{i-1}$ in each case. Then, we consider
	        the graph $Q^{2,1}_{i-1}$ obtained from $Q^{2}_{i-1}$ by adding 
	        the bold edge. In Figures~\ref{fig:casei}
	        and~\ref{fig:caseiv}, the graph $Q^{2,1}_{i-1}$ is locally depicted for each case.
	        We get the desired instances by applying the basic square construction. We describe this 
	        in more details.
	        For instance i): with the notation of Figure~\ref{fig:basic_square_0} and~\ref{fig:casei},
		it is enough to consider $e_a =x^1_1w_2$, $e_b =z_1z_2$, $e_c =x^1_1x_2$ and $e_d =y_1y_2$. 
		For instance iv): with the notation of Figure~\ref{fig:basic_square_0} and~\ref{fig:casei},
		it is enough to consider $e_a =z_1z_2$, $e_b =x^1_1w_2$, $e_c =y_1y_2$ and $e_d =x^1_1x_2$. 
		
For generating instances ii) and iii): 
			     Let $Q^{2}_{i-1}$ be the graph obtained from $Q_{i-1}$ by expanding the vertex $x_1$ 
	        to $x^1_1v_1x^2_1$ in such a way that the partition associated with the vertex $x^2_1$ is $\{y_1,z_1\}$. 	
	        Then, we add the new edge  $v_1z_2$ if we are generating instance~ii), 
	        or $v_1y_2$ if we are generating instance~iii). Then, we consider
	        the graph $Q^{2,1}_{i-1}$ obtained from $Q^{2}_{i-1}$ by adding the new edge  
	        $x^1_1y_1$ if we are generating instance ii) 
	        or $x^1_1z_1$ if we are generating instance iii).	
	        In Figures~\ref{fig:caseii}
	        and~\ref{fig:caseiii}, the graph $Q^{2,1}_{i-1}$ is locally depicted for each case.
	        For instance ii): with the notation of Figure~\ref{fig:basic_square_0} and~\ref{fig:casei},
		it is enough to consider $e_a =x^1_1w_2$, $e_b =z_1z_2$, $e_c =y_1y_2$ and $e_d =x^1_1x_2$. 
		For instance iii): with the notation of Figure~\ref{fig:basic_square_0} and~\ref{fig:casei},
		it is enough to consider $e_a =x^1_1x_2$, $e_b =y_1y_2 $, $e_c =x^1_1w_2$ and $e_d =z_1z_2$. 
\begin{figure}[h]
  \centering 
    \subfigure[ ]
  { 
    \ifpdf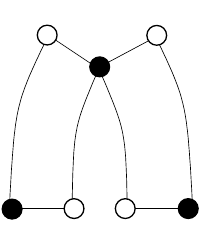\else\input{case_x1y1z1.ps_t}\fi
    \label{fig:starting_case}
  } \qquad
    \subfigure[$Q^{2,1}_{i-1}$ for i)]
  { 
    \ifpdf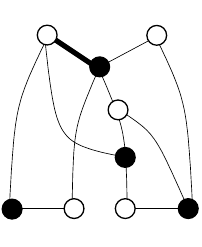\else\input{case_x1y1z1_a.ps_t}\fi
    \label{fig:casei}
  } \qquad
  \subfigure[$Q^{2,1}_{i-1}$ for ii)]
  { 
    \ifpdf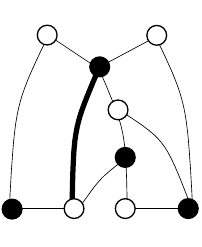\else\input{case_x1y1z1_b.ps_t}\fi
    \label{fig:caseii}
  } \qquad
    \subfigure[$Q^{2,1}_{i-1}$ for iii)]
  { 
    \ifpdf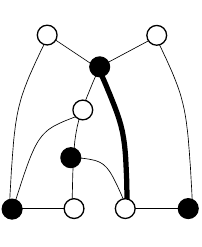\else\input{case_x1y1z1_c.ps_t}\fi
    \label{fig:caseiii}
  } \qquad
    \subfigure[$Q^{2,1}_{i-1}$ for iv)]
  { 
    \ifpdf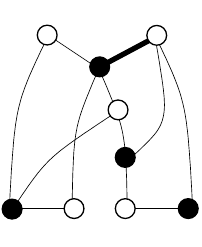\else\input{case_x1y1z1_d.ps_t}\fi
    \label{fig:caseiv}
  } 
  \caption{Local view of $Q^{2,1}_{i-1}$ in the generation of instances i), ii), iii) and iv) for configurations
  2, 5 and~7. In each configuration, we may possible have $y_2 = z_2$. In (a), for obtaining Configurations 3 
  and 8 it is enough to identify $x_2$ and $w_2$, and delete multiple edges.}
  \label{fig:caseab_degree4}
  \end{figure}
	        
   \item[Configurations 3 and 8] are respectively depicted in Figu\-res~\ref{case3} and~\ref{case8}.
   Note that with the notation of Figure~\ref{fig:intersection_of_pp}, both configurations
   may be obtained from Figure~\ref{fig:starting_case} by identifying 
   $x_2, w_2$ and by removing the double edge. Therefore, the reasoning for configurations 2, 5 and 7 applies also 
   for configurations 3 and 8.

  \item[Configurations 4 and 6. ]
These configurations are respectively depicted
 in Figures~\ref{case4} and~\ref{case6}.
With the notation of Figure~\ref{fig:intersection_of_pp}, both configurations can be
locally depicted as in Figure~\ref{fig:starting_casex1w2}.
Moreover, using the symmetry of both configurations 4 and 6, without loss of generality we can assume that
either the degree of $x_1$ and $w_2$ in $Q_{i-1}$ is $4$
or the degree of $x_1$ and $x_2$ in $Q_{i-1}$ is $4$.  Therefore, in either case 
we are allowed to expand $x_1$.
Next we show how to generate each instances i), ii), iii) and iv) of Figure~\ref{fig:cases_matching-related_1}.

Generation of instance i): 
		Let $Q^{2}_{i-1}$ be the graph obtained from $Q_{i-1}$ by expanding the vertex $x_1$ 
	        to $x^1_1u_1x^2_1$ in such a way that the partition associated with the vertex $x^2_1$ is 
	        $\{w_2,y_1\}$. Then, we add the new edge $u_1y_2$.  
	        Consider the Figure~\ref{fig:casei_46} without the bold edge for a local ilustration of $Q^{2}_{i-1}$.
	        Then, we consider the graph $Q^{2,1}_{i-1}$ obtained from $Q^{2}_{i-1}$ by adding the new edge $x^1_1w_2$,
	        namely, the bold edge of Figure~\ref{fig:casei_46}.
	        We finally obtain the desired instance i) by applying the basic square construction
	        in the same fashion as for the case of Configurations 2, 5, and 7.

Generation of instances ii) and iv):  
		Let $Q^{2}_{i-1}$ be the graph obtained from $Q_{i-1}$ by expanding the vertex $x_1$ 
	        to $x^1_1u_1x^2_1$ in such a way that the partition associated with the vertex $x^2_1$ is 
	        $\{x_2,y_1\}$. Then we add the new edge $u_1w_1$ if we are generating instance ii),
	        or $u_1z_2$ if we are generating instance iv).    
	        Then, we consider the graph $Q^{2,1}_{i-1}$ obtained from $Q^{2}_{i-1}$ adding the new edge $x^1_1x_2$
	        (bold edge in Figures~\ref{fig:caseii_46} and~\ref{fig:caseiv_46}). 
	        In Figures~\ref{fig:caseii_46} and~\ref{fig:caseiv_46}, the graph $Q^{2,1}_{i-1}$ for the
	        generation of both instances is locally depicted. Again, we get the desired instances ii) and iv) 
	        by applying the basic square construction in the same fashion as for the case of Configurations 2, 5, and 7.

\begin{figure}[h]
  \centering 
    \subfigure[]
  { 
    \ifpdf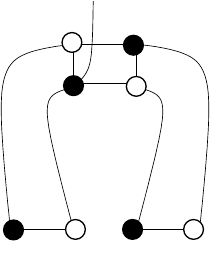\else\input{case_x1w2.ps_t}\fi
    \label{fig:starting_casex1w2}
  } \qquad
    \subfigure[$Q^{2,1}_{i-1}$ for i)]
  { 
    \ifpdf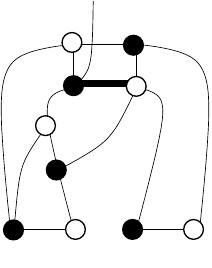\else\input{case_x1w2_a.ps_t}\fi
    \label{fig:casei_46}
  } \qquad
  \subfigure[$Q^{2,1}_{i-1}$ for ii)]
  { 
    \ifpdf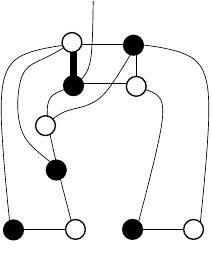\else\input{case_x1w2_b.ps_t}\fi
    \label{fig:caseii_46}
  } \qquad
    \subfigure[$Q^{2,1^*}_{i-1}$ for iii)]
  { 
    \ifpdf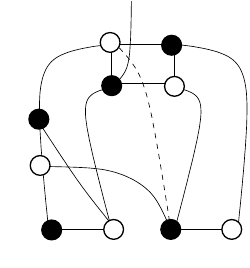\else\input{case_x1w2_c.ps_t}\fi
    \label{fig:caseiii_46}
  } \qquad
    \subfigure[$Q^{2,1}_{i-1}$ for iv)]
  { 
    \ifpdf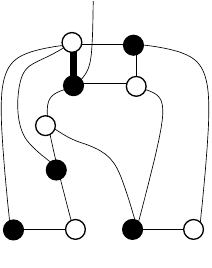\else\input{case_x1w2_d.ps_t}\fi
    \label{fig:caseiv_46}
  } 
  \caption{Local view of $Q^{2,1}_{i-1}$ or $Q^{2,1^*}_{i-1}$ in the generation of instances i), ii), iii) and iv) for configurations
  4 and 6. In each configuration, we have that $y_2 \neq z_2$ and $y_1 \neq z_1$.
  In case (d), the edge $x^1_2z_2$ may exist.} 
   \label{fig:caseab_degree4}
  \end{figure}	   

Generation of instance iii): If the edge $x_2z_2 \notin E(Q_{i-1})$, then add $x_2z_2$. We denote by $Q^{1}_{i-1}$ 
		either the graph obtained from $Q_{i-1}$ by adding $x_2z_2$ or, the graph  $Q_{i-1}$ such that 
		 $x_2z_2 \in E(Q_{i-1})$. Hence, $z_2$, $y_2$ are neighbors of $x_2$ in $Q^{1}_{i-1}$
		and clearly $y_2 \neq z_2$ (see Figures~\ref{case4} and~\ref{case6}).
		Let $Q^{2,1}_{i-1}$ be the graph obtained 
		from $Q^{1}_{i-1}$ by expanding the vertex $x_2$ to $x^1_2u_1x^2_2$ in such a way that the partition 
		associated with the vertex $x^2_2$ is 
	        $\{z_2,y_2\}$. Then, we add the new edge $u_1y_1$.
	        Then we obtain $Q^{2,1^*}_{i-1}$ in the following way:
	        if the edge $x_2z_2 \in E(Q_{i-1})$, then we obtain $Q^{2,1^*}_{i-1}$ 
	        from $Q^{2,1}_{i-1}$ by adding the new edge $x^1_2z_2$.
	        Otherwise, $Q^{2,1^*}_{i-1}= Q^{2,1}_{i-1}$.
	        In Figure~\ref{fig:caseiii_46} the graph $Q^{2,1^*}_{i-1}$ is locally depicted. 
	        Again, we use the basic square construction
	        to complete the generation.

	 \item[Configuration 1.] This configuration is depicted in Figure~\ref{case1}. 
	 By the symmetry of configuration 1 it suffices to show that 
	 we can generate instance iii); this can be generated in the same
	 fashion as the previous instance iii) for configurations 4 and 6.


	 \item[Configuration 9] is depicted in Figure~\ref{case9}. To make things easier, 
we depict in Figure~\ref{fig:onefatedge_match-match} the subgraphs that we want to generate
from Configuration 9; they correspond to the instances i$'$), ii$'$), iii$'$) and iv$'$) of Figure~\ref{fig:cases_config9}.

\begin{figure}[h]
  \centering 
    \subfigure[Instance i')]
  { 
    \ifpdf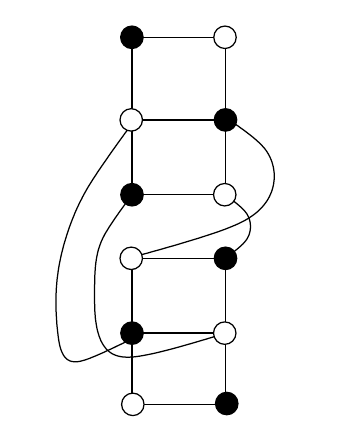\else\input{fatedge_match1_1.ps_t}\fi
    \label{fig:fatedge_1-b}
  } 
      \subfigure[Instance ii') and iv')]
  { 
    \ifpdf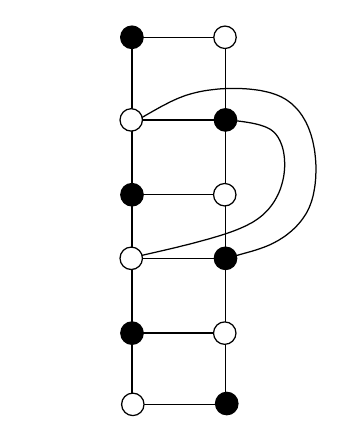\else\input{fatedge_match2.ps_t}\fi
    \label{fig:fatedge_2}
  } 
    \subfigure[Instance iii') and iv')]
  { 
    \ifpdf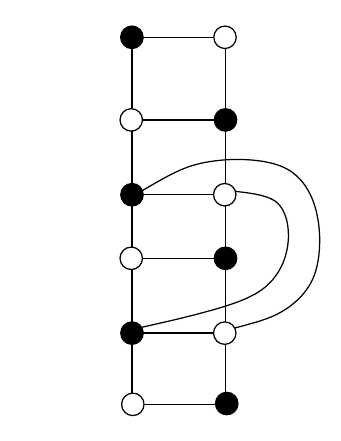\else\input{fatedge_match3.ps_t}\fi
    \label{fig:fatedge_3}
  } 
  \caption{Instances i'), ii'), iii') and iv') of Figure~\ref{fig:cases_config9} for configuration 9.}
  \label{fig:onefatedge_match-match}
  \end{figure}

	      We first focus on the generation of the configurations depicted in Figure~\ref{fig:fatedge_2}
	      and Figure~\ref{fig:fatedge_3}. By symmetry, it suffices to generate only one of them, 
	      say we generate the configuration
	      depicted in Figure~\ref{fig:fatedge_3}.

	      We split this case into two subcases: (*) at least one vertex of $\{x_1,x_2\}$ has degree $4$
	      in $Q_{i-1}$ and (**) $x_1$ and $x_2$ have degree 3 in $Q_{i-1}$.
	      
	      \begin{description}
	       \item{subcase (*):} 
without loss of generality we assume that $x_1$ has degree 4 in $Q_{i-1}$. Consider
	      the graph $Q^{2,1}_{i-1}$ obtained from $Q_{i-1}$ by expanding $x_1$ to $x^1_1ux^2_1$ in such a way that 
	      the partition associated with the vertex $x^2_1$ is $\{x_2,y_1\}$. Then we add the new edges 
	      $uy_2$ and $x_2x^1_1$. Next, we consider the graph $Q^{2,1,2}_{i-1}$ 
	      obtained from $Q^{2,1}_{i-1}$ by expanding $y_2$ to $y^1_2vy^2_2$ in such a way that 
	      the partition associated with the vertex $y^2_2$ is $\{u,x_2\}$. Then we add the new edge 
	      $vx^2_1$.  Furthermore, let $Q^{2,1,2,2}_{i-1}$ be the graph obtained from 
	      $Q^{2,1,2}_{i-1}$ by expanding $x_2$ to $x^1_2u'x^2_2$ in such a way that 
	      the partition associated with the vertex $x^2_2$ is $\{y^2_2,x^2_1\}$. Then we add the new edge 
	      $uu'$.  We finally consider the graph $Q^{2,1,2,2,2,1}_{i-1}$ 
	      obtained from $Q^{2,1,2,2}_{i-1}$ by expanding $u$ to $u^1 w u^2$ in such a way that 
	      the partition associated with the vertex $u^2$ is $\{u',x^1_1\}$. Then we add the new edges 
	      $wx^2_2$ and $wv$. The graph $Q^{2,1,2,2,2,1}_{i-1}$ is locally equal to
	      the subgraph depicted in Figure~\ref{fig:fatedge_3}.    
	     
\item{subcase (**):} we recall that the set of vertices 
	      of the square $s_x$ is given by
	      $\{x_i: i\in\mathbb{Z}_4\}$.
	       Then, by Proposition~\ref{lemma:degrees}
	      the vertices $x_0$ and $x_3$ have degree 4.  Let $Q^{2,1}_{i-1}$ be the graph obtained 
	      from $Q_{i-1}$ by expanding $x_3$ to $x^1_3ux^2_3$ in such a way that 
	    the partition associated with the vertex $x^2_3$ is $\{x_0,x_2\}$. Then we add the new edges 
	    $ux_1$ and $x^1_3x_0$. Next, we consider $Q^{2,1,2}_{i-1}$ the graph obtained from 
	    $Q^{2,1}_{i-1}$ by expanding $x_0$ to $x^1_0vw$ in such a way that 
	    the partition associated with the vertex $w$ is $\{x_1,x^2_3\}$. Then we add the new edge 
	    $vu$. Then, we consider $Q^{2,1,2,2}_{i-1}$ the graph obtained from 
	    $Q^{2,1,2}_{i-1}$ by expanding $u$ to $u^1 z u^2$ in such a way that 
	    the partition associated with the vertex $u^2$ is $\{x^2_3,x^1\}$. Then we add the new edge 
	    $zw$. We now consider $Q^{2,1,2,2,2,1}_{i-1}$ the graph obtained from 
	    $Q^{2,1,2,2}_{i-1}$ by expanding $w$ to $w^1v'w^2$ in such a way that 
	    the partition associated with the vertex $w^2$ is $\{v,z\}$. Then we add the new edges 
	    $u^2 v'$ and $x_2v'$. The graph $Q^{2,1,2,2,2,1}_{i-1}$ contains the desired instance
	    (see Figure~\ref{fig:fatedge_3}).

 \end{description}

We now show the generation of the configuration depicted in Figure~\ref{fig:fatedge_1-b}.
Again we split this case into two subcases: (*) at least one vertex in  $\{x_1,x_2,y_1,y_2\}$ has degree 4 in $Q_{i-1}$
and (**) all vertices in  $\{x_1,x_2,y_1,y_2\}$ have degree 3 in $Q_{i-1}$.
 
 \begin{description}
  \item{subcase (*):}
	    without loss of generality we suppose that $x_1$ has degree 4 in $Q_{i-1}$. 
	    We now consider $Q^{2,1}_{i-1}$ the graph obtained from 
	    $Q_{i-1}$ by expanding $x_1$ to $x^1_1ux^2_1$ in such a way that 
	    the partition associated with the vertex $x^2_1$ is $\{x_2,y_1\}$. Then we add the new edges 
	    $uy_2$ and $x^1_1x_2$. Let $Q^{2,1,2}_{i-1}$ be the graph obtained from $Q^{2,1}_{i-1}$ 
	    by expanding $x_2$ to $x^1_2u'x^2_2$ in such a way that 
	    the partition associated with the vertex $x^2_2$ is $\{x^2_1,y_2\}$. Then we add the new edge 
	    $uu'$. We shall consider $Q^{2,1,2,2}_{i-1}$ obtained from $Q^{2,1,2}_{i-1}$ 
	    by expanding $y_2$ to $y^1_2vy^2_2$ in such a way that 
	    the partition associated with the vertex $y^2_2$ is $\{u,x^2_2\}$. Then we add the new edge 
	    $vx^2_1$. Let $Q^{2,1,2,2,2,1}_{i-1}$ be the graph obtained from $Q^{2,1,2,2}_{i-1}$ 
	    expanding $u$ to $u^1 v' u^2$ in such a way that 
	    the partition associated with the vertex $u^2$ is $\{u',y^2_2\}$. Then we add the new edges 
	    $v'v$ and $u^1u'$. The graph $Q^{2,1,2,2,2,1}_{i-1}$ contains the desired instance (see Figure~\ref{fig:fatedge_1-b}).

  \item{subcase (**):}
	    by Proposition~\ref{lemma:degrees} we have that all vertices in $\{x_0,x_3,y_0,y_3\}$ 
	    have degree 4 in $Q_{i-1}$. 
	    We consider $Q^{2}_{i-1}$ the graph obtained from $Q_{i-1}$ 
	    by expanding $x_0$ to $x^1_0 u x^2_0$ in such a way that 
	    the partition associated with the vertex $x^2_0$ is $\{x_1,x_3\}$. Then we add the new edges 
	    $ux_2$ and $x^1_0x_3$.
	     We consider $Q^{2,1,2}_{i-1}$ the graph obtained from $Q^{2,1}_{i-1}$ by
	    expanding $x_3$ to $x^1_3 u' x^2_3$ in such a way that 
	    the partition associated with the vertex $x^2_3$ is $\{x^2_0,x_2\}$. Then we add the new edge 
	    $uu'$. Let  $Q^{2,1,2,2}_{i-1}$ be the graph obtained from $Q^{2,1,2}_{i-1}$ by
	    expanding $u$ to $u^1 v u^2$ in such a way that 
	    the partition associated with the vertex $u^2$ is $\{x^2_0,x_2\}$. Then we add the new edge 
	    $vx^2_3$. We consider  $Q^{2,1,2,2,2,1}_{i-1}$ the graph obtained from $Q^{2,1,2,2}_{i-1}$ by 
	    expanding $x_2$ to $x^1_2 v' x^2_2$ in such a way that 
	    the partition associated with the vertex $x^2_2$ is $\{u^2,x_1\}$. Then we add the new edge 
	    $vv'$ and $x_1x^1_2$.  The graph $Q^{2,1,2,2,2,1}_{i-1}$ contains the desired instance.

 \end{description}

  \end{description}


\begin{thebibliography}{1}

\bibitem{Jaeger19851}
F.~Jaeger.
\newblock A survey of the cycle double cover conjecture.
\newblock In B.R. Alspach and C.D. Godsil, editors, {\em Annals of Discrete
  Mathematics 27 Cycles in Graphs}, volume 115 of {\em North-Holland
  Mathematics Studies}, pages 1 -- 12. North-Holland, 1985.

\bibitem{kenyon}
R.~Kenyon.
\newblock The laplacian and dirac operators on critical planar graphs.
\newblock {\em Inventiones mathematicae}, 150(2):409--439, 2002.

\bibitem{Lovasz:1987:MSM:30820.30826}
L.~Lov\'asz.
\newblock Matching structure and the matching lattice.
\newblock {\em Journal of Combinatorial Theory, Series B}, 43(2):187 -- 222,
  1987.

\bibitem{B:lovasz-plummer}
L.~Lov\'asz and M.D. Plummer.
\newblock {\em Matching Theory}.
\newblock Akad\'emiai Kiad\'o, Budapest, 1986.
\newblock Also published as Vol. 121 of the North-Holland Mathematics Studies,
  North-Holland Publishing, Amsterdam.

\bibitem{McCuaig98}
W.~McCuaig.
\newblock Brace generation.
\newblock {\em Journal of Graph Theory}, 38(3):124--169, 2001.

\bibitem{mercat}
Ch. Mercat.
\newblock Discrete riemann surfaces and the ising model.
\newblock {\em Communications in Mathematical Physics}, 218(1):177--216, 2001.

\bibitem{opac-b1131924}
B.~Mohar and C.~Thomassen.
\newblock {\em Graphs on surfaces}.
\newblock Johns Hopkins studies in the mathematical sciences. Johns Hopkins
  University Press, Baltimore (MD), London, 2001.

\bibitem{Zhang97integerflows}
C.-Q. Zhang.
\newblock {\em Integer flows and cycle covers of graphs}.
\newblock Marcel Dekker, Inc, 1997.

\end{thebibliography}
\end{document}